\documentclass[12pt]{amsart}
\usepackage{amsfonts,amsmath,a4wide,latexsym,amssymb,amscd,amsthm,amsxtra,mathrsfs,amsrefs}

\usepackage[bookmarksnumbered, colorlinks, plainpages]{hyperref}
\hypersetup{colorlinks=true,linkcolor=red, anchorcolor=green, citecolor=cyan, urlcolor=red, filecolor=magenta, pdftoolbar=true}

\usepackage{upref}
\usepackage{txfonts}
\usepackage{graphicx}

\usepackage[width=2cm, height=2cm, left=2cm, right=2cm, top=2cm, bottom=1cm]{geometry}

\allowdisplaybreaks

\newtheorem{theorem}{Theorem}[section]

\newtheorem{corollary}[theorem]{Corollary}
\theoremstyle{definition}

\newtheorem{convention}[theorem]{Convention}

\theoremstyle{remark}
\newtheorem{remark}[theorem]{Remark}
\numberwithin{equation}{section}

\def\esup{\operatornamewithlimits{ess\,sup}}

\begin{document}

\baselineskip=17pt

\title[Weighted iterated Hardy-type inequalities]{Weighted iterated Hardy-type inequalities}

\author[A. Gogatishvili]{Amiran Gogatishvili}
\address{Institute of Mathematics \\
Academy of Sciences of the Czech Republic \\
\v Zitn\'a~25 \\
115~67 Praha~1, Czech Republic} \email{gogatish@math.cas.cz}

\author[R.Ch.Mustafayev]{Rza Mustafayev}
\address{Department of Mathematics \\ Faculty of Science and Arts \\ Kirikkale
University \\ 71450 Yahsihan, Kirikkale, Turkey}
\email{rzamustafayev@gmail.com}

\thanks{The research of A. Gogatishvili was partly supported by the grants P201-13-14743S
of the Grant Agency of the Czech Republic and RVO: 67985840, by
Shota Rustaveli National Science Foundation grants no. 31/48
(Operators in some function spaces and their applications in Fourier
Analysis) and no. DI/9/5-100/13 (Function spaces, weighted
inequalities for integral operators and problems of summability of
Fourier series). The research of both authors was partly supported
by the joint project between  Academy of Sciences of Czech Republic
and The Scientific and Technological Research Council of Turkey}



\begin{abstract}
In this paper a reduction and equivalence theorems for the
boundedness of the composition of a quasilinear operator $T$ with
the Hardy and Copson operators in weighted Lebesgue spaces are
proved. New equivalence theorems are obtained for the operator $T$
to be bounded in weighted Lebesgue spaces restricted to the cones of
monotone functions,  which allow to change the cone of
non-decreasing functions to the cone of non-increasing functions and
vice versa not changing the operator $T$. New characterizations of
the weighted Hardy-type inequalities on the cones of monotone
functions are given. The validity of so-called weighted iterated
Hardy-type inequalities are characterized.
\end{abstract}

\maketitle

% % % % % % % % % % % % % % % % % % % % % % % % % %
% % % % % % % % % % % % % % % % % % % % % % % % % %

\section{Introduction}\label{in}

The well-known two-weight Hardy-type inequalities
\begin{equation}\label{Hardy.ineq.1}
\bigg( \int_0^{\infty} \bigg(\int_0^x f (\tau)\,d\tau\bigg)^q
w(x)\,dx\bigg)^{1 / q} \le c\bigg(\int_0^{\infty}
f^p(x)v(x)\,dx\bigg)^{1/p}
\end{equation}
and
\begin{equation}\label{Hardy.ineq.2}
\bigg( \int_0^{\infty} \bigg( \int_x^{\infty} f
(\tau)\,d\tau\bigg)^q w(x)\,dx\bigg)^{1 / q} \le
c\bigg(\int_0^{\infty} f^p(x)v(x)\,dx\bigg)^{1/p}
\end{equation}
for all non-negative measurable functions $f$ on $(0,\infty)$, where
$0 < p,\,q < \infty$ with $c$ being a constant independent of $f$,
have a broad variety of applications and represents now a basic tool
in many parts of mathematical analysis, namely in the study of
weighted function inequalities. For the results, history and
applications of this problem, see \cites{opkuf,kufpers, kufmalpers}.

Throughout the paper we assume that $I : = (a,b)\subseteq
(0,\infty)$. By ${\mathfrak M} (I)$ we denote the set of all
measurable functions on $I$. The symbol ${\mathfrak M}^+ (I)$ stands
for the collection of all $f\in{\mathfrak M} (I)$ which are
non-negative on $I$, while ${\mathfrak M}^+ (I;\downarrow)$ and
${\mathfrak M}^+ (I;\uparrow)$ are used to denote the subset of
those functions which are non-increasing and non-decreasing on $I$,
respectively. When $I = (0,\infty)$, we write simply ${\mathfrak
M}^{\downarrow}$ and ${\mathfrak M}^{\uparrow}$ instead of
${\mathfrak M}^+ (I;\downarrow)$ and ${\mathfrak M}^+ (I;\uparrow)$,
accordingly. The family of all weight functions (also called just
weights) on $I$, that is, locally integrable non-negative functions
on $(0,\infty)$, is given by ${\mathcal W}(I)$.

For $p\in (0,\infty]$ and $w\in {\mathfrak M}^+(I)$, we define the
functional $\|\cdot\|_{p,w,I}$ on ${\mathfrak M} (I)$ by
\begin{equation*}
\|f\|_{p,w,I} : = \left\{\begin{array}{cl}
\left(\int_I |f(x)|^p w(x)\,dx \right)^{1/p} & \qquad\mbox{if}\qquad p<\infty \\
\esup_{I} |f(x)|w(x) & \qquad\mbox{if}\qquad p=\infty.
\end{array}
\right.
\end{equation*}

If, in addition, $w\in {\mathcal W}(I)$, then the weighted Lebesgue
space $L^p(w,I)$ is given by
\begin{equation*}
L^p(w,I) = \{f\in {\mathfrak M} (I):\,\, \|f\|_{p,w,I} < \infty\},
\end{equation*}
and it is equipped with the quasi-norm $\|\cdot\|_{p,w,I}$.

When $w\equiv 1$ on $I$, we write simply $L^p(I)$ and
$\|\cdot\|_{p,I}$ instead of $L^p(w,I)$ and $\|\cdot\|_{p,w,I}$,
respectively.

Suppose $f$ be a measurable a.e. finite function on ${\mathbb R}^n$.
Then its non-increasing rearrangement $f^*$ is given by
$$
f^* (t) = \inf \{\lambda > 0: |\{x \in {\mathbb R}^n:\, |f(x)| >
\lambda \}| \le t\}, \quad t \in (0,\infty),
$$
and let $f^{**}$ denotes the Hardy-Littlewood maximal function of
$f$, i.e.
$$
f^{**}(t) : = \frac{1}{t} \int_0^t f^* (\tau)\,d\tau, \quad t > 0.
$$
Quite many familiar function spaces can be defined using the
non-increasing rearrangement of a function. One of the most
important classes of such spaces are the so-called classical Lorentz
spaces.

Let $p \in (0,\infty)$ and $w \in {\mathcal W}$. Then the classical
Lorentz spaces $\Lambda^p (w)$ and $\Gamma^p (w)$ consist of all
functions $f \in {\mathfrak M}$ for which  $\|f\|_{\Lambda^p(w)} <
\infty$ and $\|f\|_{\Gamma^p(w)} < \infty$, respectively. Here it is
$$
\|f\|_{\Lambda^p(w)} : = \|f^*\|_{p,w,(0,\infty)} \qquad \mbox{and}
\qquad \|f\|_{\Gamma^p(w)} : = \|f^{**}\|_{p,w,(0,\infty)}.
$$
For more information about the Lorentz $\Lambda$ and $\Gamma$ see
e.g. \cite{cpss} and the references therein.

There has been considerable progress in the circle of problems
concerning characterization of boundedness of classical operators
acting in weighted Lorentz spaces since the beginnig of the 1990s.
The first results on the problem $\Lambda^p(v) \hookrightarrow
\Gamma^p(v)$, $1 < p < \infty$, which is equivalent to inequality
\eqref{Hardy.ineq.1} restricted to the cones of non-increasing
functions, were obtained by Boyd \cite{boyd} and in an explicit form
by Ari{\~n}o and Muckenhoupt \cite{arinomuck}. The problem with $w
\neq v$ and $p\neq q$, $1 < p,\,q < \infty$ was first successfully
solved by Sawyer \cite{sawyer1990}. Many articles on this topic
followed, providing the results for a wider range of parameters. In
particular, much attention was paid to inequalities
\eqref{Hardy.ineq.1} and \eqref{Hardy.ineq.2} restricted to the
cones of monotone functions; see for instance
\cites{arinomuck,sawyer1990,steptrans,step1993,carsor1993,heinstep1993,ss,Sinn,popo,gogpick2000,bengros,gogpick2007,gjop,cgmp2008,gogstepdokl2012_1,gogstepdokl2012_2,GogStep1,GogStep,gold2011.1,gold2011.2,johstepush,LaiShanzhong,gold2001},
survey \cite{cpss}, the monographs \cites{kufpers, kufmalpers}, for
the latest development of this subject see \cite{GogStep}, and
references given there. The restricted operator inequalities may
often be handled by the so-called "reduction theorems". These, in
general, reduce a restricted inequality into certain non-restricted
inequalities. For example, the restriction to non-increasing or
quasi-concave functions may be handled in this way, see e.g.
\cites{sinn2002,gogstepdokl2012_1,gogstepdokl2012_2,GogStep1,GogStep}.
At the initial stage the main tool was the Sawyer duality principle
\cite{sawyer1990}, which allowed one to reduce an $L^p - L^q$
inequality for monotone functions with $1 < p,\,q < \infty$ to a
more manageable inequality for arbitrary non-negative functions.
This principle was extended by Stepanov in \cite{steptrans} to the
case $0 < p < 1 < q < \infty$. In the same work Stepanov applied a
different approach to this problem, so-called reduction theorems,
which enabled to extend the range of parameters to $1 < p < \infty$,
$0 < q < \infty$.  The case $p \le q$, $0 < p \le 1$ was
alternatively characterized in
\cites{steptrans,step1993,carsor1993,burgold,LaiShanzhong}. Later on
some direct reduction theorems were found in
\cites{gogpick2007,cgmp2008,GogStep} involving supremum operators
which work for the case $0 < q < p \le 1$.

In this paper we consider operators $T:{\mathfrak M}^+ \rightarrow
{\mathfrak M}^+$ satisfying the following conditions:

{\rm (i)} $T(\lambda f) = \lambda Tf$ for all $\lambda \ge 0$ and $f
\in {\mathfrak M}^+$;

{\rm (ii)} $Tf(x) \le c Tg(x)$ for almost all $x \in {\mathbb R}_+$
if $f(x) \le g(x)$ for almost all $x \in {\mathbb R}_+$, with
constant $c > 0$ independent of $f$ and $g$;

{\rm (iii)} $T(f+ \lambda {\bf 1}) \le c (T f + \lambda T {\bf 1})$
for all $f \in {\mathfrak M}^+$ and $\lambda \ge 0$, with a constant
$c > 0$ independent of $f$ and $\lambda$.

Given a operator $T:{\mathfrak M}^+ \rightarrow {\mathfrak M}^+$,
for $0 < p <\infty$ and $u\in{\mathfrak M}^+$, denote by
$$
T_{p,u} (g) : = ( T (g^p u))^{1/p}, \qquad g \in {\mathfrak M}^+.
$$
Hence $T_{1,{\bf 1}} \equiv T$. When $p = 1$, we write $T_{u}$
instead of $T_{1,u}$.

Denote by
$$
H g (t) : = \int_0^t g(s)\,ds, \qquad g \in {\mathfrak M}^+,
$$
and
$$
H^* g (t) : = \int_t^{\infty} g(s)\,ds, \qquad g \in {\mathfrak
M}^+,
$$
the Hardy operator and Copson operator, respectively.

In the paper we prove a reduction and equivalence theorems for the
boundedness of the composition operators $T \circ H$ or $T \circ
H^*$ of a quasiliear operator $T: {\mathfrak M}^+ \rightarrow
{\mathfrak M}^+$ with the operators $H$ and $H^*$ in weighted
Lebesgue spaces.  To be more precise, we consider inequalities
\begin{equation}
\label{RT.thm.main.3.eq.1} \left\|T \bigg(\int_0^x
h\bigg)\right\|_{\beta,w,(0,\infty)} \leq c
\,\|h\|_{s,v,(0,\infty)}, \, h \in {\mathfrak M}^+,
\end{equation}
and
\begin{equation}
\label{RT.thm.main.4.eq.1} \left\|T \bigg(\int_x^{\infty}
h\bigg)\right\|_{\beta,w,(0,\infty)} \leq c
\,\|h\|_{s,v,(0,\infty)}, \, h \in {\mathfrak M}^+.
\end{equation}
Using these equivalence theorems, in particular, we completely
characterize the validity of the iterated Hardy-type inequalities
\begin{equation}\label{IHI.1}
\left\|H_{p,u} \bigg( \int_0^x h\bigg)\right\|_{q,w,(0,\infty)} \leq
c \,\|h\|_{s,v,(0,\infty)},
\end{equation}
and
\begin{equation}\label{IHI.2}
\left\|H_{p,u} \bigg( \int_x^{\infty}
h\bigg)\right\|_{q,w,(0,\infty)} \leq c \,\|h\|_{s,v,(0,\infty)},
\end{equation}
where $0<p<\infty$, $0 < q \le \infty$, $1 \le s < \infty$, $u$, $w$
and $v$ are weight functions on  $(0,\infty)$.

It is worth to mentoin that the characterizations of "dual"
inequalities
\begin{equation}\label{IHI.3}
\left\| H_{p,u}^* \bigg( \int_x^{\infty}
h\bigg)\right\|_{q,w,(0,\infty)} \leq c \,\|h\|_{s,v,(0,\infty)},
\end{equation}
and
\begin{equation}\label{IHI.4}
\left\|H_{p,u}^* \bigg( \int_0^x h\bigg)\right\|_{q,w,(0,\infty)}
\leq c \,\|h\|_{s,v,(0,\infty)},
\end{equation}
can be easily obtained  from the solutions of inequalities
\eqref{IHI.1} - \eqref{IHI.2}, respectively, by change of variables.

In the case when $p=1$, using the Fubini Theorem, inequalities
\eqref{IHI.1} and \eqref{IHI.2} can be reduced to the weighted $L^s
- L^q$ boundedness problem of the  Volterra operator
$$
(Kh)(x) : = \int_0^x k(x,y) h(y)\,dy, \quad x > 0,
$$
with the kernel
$$
k(x,y) : = \int_y^x u(t)\,dt, \quad 0 < y \le x < \infty,
$$
and the Stieltjes operator
$$
(Sh)(x) = \int_0^{\infty} \frac{h(t)\,dt}{U(x) + U(t)},
$$
respectively, and consequently, can be easily solved. Indeed:

By  the Fubini Theorem, we see that
$$
\int_0^x \left( \int_0^t h(\tau)\,d\tau\right)u(t)\,dt = \int_0^x
k(x,\tau) h(\tau)\,d\tau, \qquad h\in {\mathfrak M}^+(0,\infty).
$$

On the other hand,  it is easy to see that
$$
\int_0^x \left( \int_t^{\infty} h(s)\,ds\right)u(t)\,dt \approx
U(x)\cdot S(hU)(x), \qquad h\in {\mathfrak M}^+(0,\infty).
$$

Note that the weighted $L^s - L^q$ boundedness of Volterra operators
$K$, that is, inequality
\begin{equation}\label{volter}
\|Kh\|_{q,w,(0,\infty)} \le c \|h\|_{s,v,(0,\infty)},\quad h \in
{\mathfrak M}^+(0,\infty),
\end{equation}
is completely characterized for $1 \le s \le \infty$, $0 <  q \le
\infty$ (see \cite{GogStep} and references given there).

The usual Stieltjes transform is obtained on putting $U(x) \equiv
x$. In the case $U(x) \equiv x^{\lambda}$, $\lambda > 0$, the
boundedness of the operator $S$ between weighted $L^s$ and $L^q$
spaces, namely inequality
\begin{equation}\label{stielt.1}
\|Sh\|_{q,w,(0,\infty)} \le c \|h\|_{s,v,(0,\infty)},\quad h \in
{\mathfrak M}^+(0,\infty),
\end{equation}
was investigated in \cite{ander} (when $1\le s \le q \le \infty$),
in \cite{sin1988} (when $1\le q < s \le \infty$), in \cite{gop2009}
(see also \cite{g1}) (when $1 < s < \infty$, $0 < q \le \infty$),
where the result is presented without proof. This problem also was
considered in \cite{gogkufpers2009} and \cites{gpswdokl,gpsw2014},
where completely different approach was used, based on the so called
``gluing lemma'' (see also \cite{gogkufpers2013}). It is proved in
\cite{GogMusPers2} (when $1 \le s \le \infty$, $0 < q \le \infty$)
that inequality \eqref{stielt.1} holds if and only if
\begin{equation}\label{stielt.2}
\bigg\| H_u \bigg( \int_x^{\infty} h \bigg)\bigg\|_{q,w,(0,\infty)}
\le c \|hU\|_{s,v,(0,\infty)}, \quad h \in {\mathfrak
M}^+(0,\infty),
\end{equation}
holds, and the solution of \eqref{stielt.1} is obtained using
characterization of inequality \eqref{stielt.2}.

Note that inequality \eqref{IHI.2} has been completely characterized
in \cite{GogMusPers1} and \cite{GogMusPers2} in the case
$0<p<\infty$, $0<q\leq \infty$, $1 \le s \le \infty$ by using
difficult discretization and anti-discretization methods.
Inequalities \eqref{IHI.1} - \eqref{IHI.2} and \eqref{IHI.3} -
\eqref{IHI.4} were considered also in \cite{ProkhStep1} and
\cite{ProkhStep2}, but characterization obtained there is not
complete and seems to us unsatisfactory from a practical point of
view.

We pronounce that the characterizations of inequalities
\eqref{IHI.1}-\eqref{IHI.2} and \eqref{IHI.3}-\eqref{IHI.4} are
important because many inequalities for classical operators  can be
reduced to them (for illustrations of this important fact, see, for
instance, \cite{GogMusPers2}). These inequalities play an important
role in the theory of Morrey-type spaces and other topics (see
\cite{BGGM1}, \cite{BGGM2} and \cite{BO}). It is worth to mention
that using characterizations of weighted Hardy inequalities we can
show that the characterization of the boundedness of bilinear Hardy
inequalities, namely of the inequality
\begin{equation}
\| T_1 f \cdot T_2 g \|_{w,q,(0,\infty)} \le c
\|f\|_{p_1,v_1,(0,\infty)} \|g\|_{p_2,v_2,(0,\infty)},
\end{equation}
for all $f \in L^{p_1}(v_1,(0,\infty))$ and $g \in
L^{p_2}(v_2,(0,\infty))$ with constant $c$ independent of $f$ and
$g$, where $T_i = H \,\mbox{or}\, H^*$, $i = 1,\,2$, are equivalent
to inequalities \eqref{IHI.1}-\eqref{IHI.2} and
\eqref{IHI.3}-\eqref{IHI.4} (see, for instance, \cite{aor}).

It is well-known that when $T$ is a integral operator then by
substitution of variables it is possible to change the cone of
non-decreasing functions to the cone of non-increasing functions and
vice versa, when considering inequalities
\begin{equation}\label{TonMonCon.1}
\|Tf \|_{\beta,w,(0,\infty)} \le c \| f \|_{s,v,(0,\infty)}, \qquad
f \in {\mathfrak M}^{\downarrow}(0,\infty),
\end{equation}
and
\begin{equation}\label{TonMonCon.2}
\|Tf \|_{\beta,w,(0,\infty)} \le c \| f \|_{s,v,(0,\infty)}, \qquad
f \in {\mathfrak M}^{\uparrow}(0,\infty),
\end{equation}
but this procedure changes $T$ also as usually to the "dual"
operator. Theorems proved in Section \ref{WeightedIneqOnTheCone}
allows to change the cones to each other not changing the operator
$T$. This new observation enables to state that if we know solution
of one inequality on any cone of monotone functions, then we could
characterize the inequality on the other cone of monotone functions.

The paper is organized as follows. Section \ref{pre} contains some
preliminaries along with the standard ingredients used in the
proofs. In Section \ref{RT} we prove the reduction and equivalence
theorems for the boundedness of the composition operators $T \circ
H$ or $T \circ H^*$  in weighted Lebesgue spaces. In Section
\ref{WeightedIneqOnTheCone} the equivalence theorems which allow to
change the cones of monotone functions to each other not changing
the operator $T$ are proved. In Section \ref{HardyIneqOnTheCone} we
obtain a new characterizations of the weighted Hardy-type
inequalities on the cones of monotone functions. In Section
\ref{RT.SC} we give complete characterization of inequalities
\eqref{IHI.1} - \eqref{IHI.2} and \eqref{IHI.3} - \eqref{IHI.4}.

\section{Notations and Preliminaries}\label{pre}

Throughout the paper, we always denote by  $c$ or $C$ a positive
constant, which is independent of main parameters but it may vary
from line to line. However a constant with subscript or superscript
such as $c_1$ does not change in different occurrences. By
$a\lesssim b$, ($b\gtrsim a$) we mean that $a\leq \lambda b$, where
$\lambda >0$ depends on inessential parameters. If $a\lesssim b$ and
$b\lesssim a$, we write $a\approx b$ and say that $a$ and $b$ are
equivalent.  We will denote by $\bf 1$ the function ${\bf 1}(x) =
1$, $x \in (0,\infty)$. Unless a special remark is made, the
differential element $dx$ is omitted when the integrals under
consideration are the Lebesgue integrals. Everywhere in the paper,
$u$, $v$ and $w$ are weights.
\begin{convention}\label{Notat.and.prelim.conv.1.1}
    We adopt the following conventions:

    {\rm (i)} Throughout the paper we put $0 \cdot \infty = 0$, $\infty / \infty =
    0$ and $0/0 = 0$.

    {\rm (ii)} If $p\in [1,+\infty]$, we define $p'$ by $1/p + 1/p' = 1$.

    {\rm (iii)} If $0 < q < p < \infty$, we define $r$ by $1 / r  = 1 / q - 1 / p$.

    {\rm (iv)} If $I = (a,b) \subseteq {\mathbb R}$ and $g$ is monotone
    function on $I$, then by $g(a)$ and $g(b)$ we mean the limits
    $\lim_{x\rightarrow a+}g(x)$ and $\lim_{x\rightarrow b-}g(x)$, respectively.
\end{convention}

To state the next statements we need the following notations:
$$
\begin{array}{ll}
U(t) : = \int_0^t u, & U_*(t) : = \int_t^{\infty} u,  \\[10pt]
V(t) : = \int_0^t v, & V_*(t) : = \int_t^{\infty} v,\\ [10pt] W(t) :
= \int_0^t w, & W_*(t) :  = \int_t^{\infty} w.
\end{array}
$$
\begin{theorem}[\cite{GogStep}, Theorem 3.1]\label{Reduction.Theorem.thm.3.1}
    Let $0 < \beta \le \infty$ and $1 \le s < \infty$, and let $T: {\mathfrak M}^+ \rightarrow {\mathfrak M}^+$
    be a positive operator. Then the inequality
    \begin{equation}\label{Reduction.Theorem.eq.1.1}
    \|Tf \|_{\beta,w,(0,\infty)} \le c \| f \|_{s,v,(0,\infty)}, \qquad f \in {\mathfrak M}^{\downarrow}(0,\infty)
    \end{equation}
    implies the inequality
    \begin{equation}\label{Reduction.Theorem.eq.1.2}
    \bigg\| T\bigg( \int_x^{\infty} h \bigg)\bigg\|_{\beta,w,(0,\infty)} \le c \| h
    \|_{s, V^{s} v^{1-s},(0,\infty)}, \qquad h \in {\mathfrak M}^+(0,\infty).
    \end{equation}
    If $V(\infty) = \infty$ and if $T$ is an operator satisfying
    conditions {\rm (i)-(ii)}, then the condition
    \eqref{Reduction.Theorem.eq.1.2} is sufficient for inequality
    \eqref{Reduction.Theorem.eq.1.1} to hold on the cone ${\mathfrak M}^{\downarrow}$.
    Further, if $0 < V(\infty) < \infty$, then a sufficient condition for
    \eqref{Reduction.Theorem.eq.1.1} to hold on ${\mathfrak M}^{\downarrow}$ is that both
    \eqref{Reduction.Theorem.eq.1.2} and
    \begin{equation}\label{Reduction.Theorem.eq.1.3}
    \|T {\bf 1}\|_{\beta,w,(0,\infty)} \le c \|{\bf 1}\|_{s,v,(0,\infty)}
    \end{equation}
    hold in the case when $T$ satisfies the conditions {\rm (i)-(iii)}.
\end{theorem}

\begin{theorem}[\cite{GogStep}, Theorem 3.2]\label{Reduction.Theorem.thm.3.2}
    Let $0 < \beta \le \infty$ and $1 \le s < \infty$, and let $T: {\mathfrak M}^+ \rightarrow {\mathfrak M}^+$
    satisfies conditions {\rm (i)} and {\rm (ii)}. Then a sufficient
    condition for inequality \eqref{Reduction.Theorem.eq.1.1} to hold is
    that
    \begin{equation}\label{Reduction.Theorem.eq.1.222}
    \bigg\| T\bigg( \frac{1}{V^2(x)} \int_0^x hV \bigg)\bigg\|_{\beta,w,(0,\infty)}
    \le c \| h \|_{s, v^{1-s},(0,\infty)}, \qquad h \in {\mathfrak M}^+(0,\infty).
    \end{equation}
    Moreover, \eqref{Reduction.Theorem.eq.1.1} is necessary for
    \eqref{Reduction.Theorem.eq.1.222} to hold if conditions {\rm
        (i)-(iii)} are all satisfied.
\end{theorem}

\begin{theorem}[\cite{GogStep}, Theorem 3.3]\label{Reduction.Theorem.thm.3.3}
    Let $0 < \beta \le \infty$ and $1 \le s < \infty$, and let $T: {\mathfrak M}^+ \rightarrow {\mathfrak M}^+$
    be a positive operator. Then the inequality
    \begin{equation}\label{Reduction.Theorem.eq.1.1.00}
    \|Tf \|_{\beta,w,(0,\infty)} \le c \| f \|_{s,v,(0,\infty)}, \qquad f \in {\mathfrak M}^{\uparrow}(0,\infty)
    \end{equation}
    implies the inequality
    \begin{equation}\label{Reduction.Theorem.eq.1.2.00}
    \bigg\| T\bigg( \int_0^x h \bigg)\bigg\|_{\beta,w,(0,\infty)} \le c \| h \|_{s,
        V_*^{s} v^{1-s},(0,\infty)}, \qquad h \in {\mathfrak M}^+(0,\infty).
    \end{equation}
    If $V_*(0) = \infty$ and if $T$ is an operator satisfying the
    conditions {\rm (i)-(ii)}, then the condition
    \eqref{Reduction.Theorem.eq.1.2.00} is sufficient for inequality
    \eqref{Reduction.Theorem.eq.1.1.00} to hold. If $0 < V_*(0) < \infty$
    and $T$ is an operator satisfying the conditions {\rm
        (i)-(iii)}, then \eqref{Reduction.Theorem.eq.1.1.00} follows from
    \eqref{Reduction.Theorem.eq.1.2.00} and
    \eqref{Reduction.Theorem.eq.1.3}.
\end{theorem}

\begin{theorem}[\cite{GogStep}, Theorem 3.4]\label{Reduction.Theorem.thm.3.4}
    Let $0 < \beta \le \infty$ and $1 \le s < \infty$, and let $T: {\mathfrak M}^+ \rightarrow {\mathfrak M}^+$
    satisfies conditions {\rm (i)} and {\rm (ii)}. Then a sufficient
    condition for inequality \eqref{Reduction.Theorem.eq.1.1.00} to hold
    is that
    \begin{equation}\label{Reduction.Theorem.eq.1.2222}
    \bigg\| T\bigg( \frac{1}{V_*^2(x)} \int_x^{\infty} hV_*
    \bigg)\bigg\|_{\beta,w,(0,\infty)} \le c \| h \|_{s, v^{1-s},(0,\infty)}, \qquad h \in
    {\mathfrak M}^+(0,\infty).
    \end{equation}
    Moreover, \eqref{Reduction.Theorem.eq.1.1.00} is necessary for
    \eqref{Reduction.Theorem.eq.1.2222} to hold if conditions {\rm
        (i)-(iii)} are all satisfied.
\end{theorem}

% % % % % % % % % % % % % % % % % % % % % % % % % % % %

% % % % % % % % % % % % % % % % % % % % % % % % % % % %

\section{Reduction and equivalence theorems}\label{RT}

In this section we prove some reduction and  equivalence theorems
for inequalities \eqref{RT.thm.main.3.eq.1} and
\eqref{RT.thm.main.4.eq.1}.

\subsection{The case $1 < s < \infty$}

The following theorem allows to reduce the iterated inequality
\eqref{RT.thm.main.3.eq.1} to the inequality on the cone of
non-increasing functions.
\begin{theorem}\label{RT.thm.main.3}
    Let $0 < \beta \le \infty$, $1 < s < \infty$, and let $T: {\mathfrak M}^+ \rightarrow {\mathfrak M}^+$
    satisfies conditions {\rm (i)-(iii)}.  Assume that $u,\,w \in {\mathcal W}(0,\infty)$
    and $v \in {\mathcal W}(0,\infty)$ be such that
    \begin{equation}\label{RT.thm.main.3.eq.0}
    \int_0^x v^{1-s^{\prime}}(t)\,dt < \infty \qquad \mbox{for all} \qquad x > 0.
    \end{equation}
    Then inequality \eqref{RT.thm.main.3.eq.1}
    holds iff
    \begin{equation}\label{RT.thm.main.3.eq.2}
    \| T_{\Phi^2} f \|_{\beta, w, (0,\infty)} \le c \| f \|_{s, \phi,(0,\infty)}, \, f \in
    {\mathfrak M}^{\downarrow},
    \end{equation}
    holds, where
    $$
    \phi (x) \equiv \phi\big[v;s\big](x) : = \bigg( \int_0^x
    v^{1-{s}^{\prime}}(t)\,dt\bigg)^{- \frac{s^{\prime}}{s^{\prime} + 1}}
    v^{1-{s}^{\prime}}(x)
    $$
    and
    $$
    \Phi(x) \equiv \Phi \big[v;s\big](x) = \int_0^x \phi(t)\,dt = \bigg( \int_0^x
    v^{1-{s}^{\prime}}(t)\,dt\bigg)^{\frac{1}{s^{\prime} + 1}}.
    $$
\end{theorem}

\begin{proof}
    Note that $\Phi^{-s}\phi^{1-s} \approx v$. Inequality
    \eqref{RT.thm.main.3.eq.1} is equivalent to the inequality
    \begin{equation}\label{RT.thm.main.3.eq.3}
    \bigg\| T_{\Phi^2} \bigg(
    \frac{1}{\Phi^2(x)}\int_0^x h \bigg)\bigg\|_{\beta,w,(0,\infty)}
    \le c \|h\|_{s, \Phi^{-s} \phi^{1-s},(0,\infty)}, \, h \in {\mathfrak M}^+.
    \end{equation}
    Obviously, \eqref{RT.thm.main.3.eq.3} is equivalent to
    \begin{equation}\label{RT.thm.main.3.eq.4}
    \bigg\| T_{\Phi^2} \bigg(
    \frac{1}{\Phi^2(x)}\int_0^x h \Phi \bigg)\bigg\|_{\beta,w,(0,\infty)}
    \le c \|h\|_{s, \phi^{1-s},(0,\infty)}, \, h \in {\mathfrak M}^+.
    \end{equation}
    By Theorem \ref{Reduction.Theorem.thm.3.2}, inequality
    \eqref{RT.thm.main.3.eq.4}  is equivalent to
    \begin{equation*}
    \| T_{\Phi^2} f \|_{\beta, w, (0,\infty)} \le c \| f \|_{s, \phi,(0,\infty)}, \, f \in
    {\mathfrak M}^{\downarrow}.
    \end{equation*}
    This completes the proof.
\end{proof}

We immediately get the following equivalence statements.
\begin{corollary}\label{RT.cor.main.3}
    Let $0 < \beta \le \infty$, $1 < s < \infty$, $0 < \delta \le s$, and let $T: {\mathfrak M}^+ \rightarrow {\mathfrak M}^+$
    satisfies conditions {\rm (i)-(iii)}.  Assume that $u,\,w \in {\mathcal W}(0,\infty)$
    and $v \in {\mathcal W}(0,\infty)$ be such that \eqref{RT.thm.main.3.eq.0} holds. Then
    inequality \eqref{RT.thm.main.3.eq.1}
    holds iff both
    \begin{equation}\label{RT.cor.main.3.eq.2}
    \bigg\| T_{\Phi^2}  \bigg( \left\{ \int_x^{\infty} h^{\delta}\right\}^{1 / \delta}\bigg) \bigg\|_{\beta, w, (0,\infty)} \le
    c \| h \|_{s, \Phi^{s / \delta} \phi^{1 - s / \delta},(0,\infty)}, \, h \in {\mathfrak M}^+,
    \end{equation}
    and
    \begin{equation}\label{RT.cor.main.3.eq.3}
    \| T_{\Phi^2} ({\bf 1})\|_{\beta, w, (0,\infty)} \le c \| {\bf 1} \|_{s,
        \phi,(0,\infty)},
    \end{equation}
    hold.
\end{corollary}
\begin{proof}
    By Theorem \ref{RT.thm.main.3}, inequality \eqref{RT.thm.main.3.eq.1} is
    equivalent to
    \begin{equation}\label{RT.cor.main.3.eq.4}
    \| T_{\Phi^2} f \|_{\beta, w, (0,\infty)} \le c \| f \|_{s, \phi,(0,\infty)}, \, f \in
    {\mathfrak M}^{\downarrow}.
    \end{equation}
    Since \eqref{RT.cor.main.3.eq.4} is equivalent to
    \begin{equation}\label{RT.cor.main.3.eq.5}
    \| \widetilde{T} f \|_{\beta / \delta, w, (0,\infty)} \le c^{\delta} \| f \|_{s / \delta, \phi,(0,\infty)}, \, f \in
    {\mathfrak M}^{\downarrow},
    \end{equation}
    with
    $$
    \widetilde{T} (f) : = \left\{ T_{\Phi^2}(f^{1/ \delta})\right\}^{\delta},
    $$
    it remains to apply Theorem \ref{Reduction.Theorem.thm.3.1}.
\end{proof}

\begin{corollary}\label{RT.cor.main.3.1}
    Let $0 < \beta \le \infty$, $1 < s < \infty$, $0 < \delta \le s$, and let $T: {\mathfrak M}^+ \rightarrow {\mathfrak M}^+$
    satisfies conditions {\rm (i)-(iii)}.  Assume that $u,\,w \in {\mathcal W}(0,\infty)$
    and $v \in {\mathcal W}(0,\infty)$ be such that \eqref{RT.thm.main.3.eq.0} holds. Then
    inequality \eqref{RT.thm.main.3.eq.1}
    holds iff
    \begin{equation}\label{RT.cor.main.3.1.eq.2}
    \bigg\| T_{\Phi^{2(1 - 1 / \delta)}} \bigg(\bigg\{\int_0^x h^{\delta}\Phi \bigg\}^{1  / \delta}\bigg) \bigg\|_{\beta, w, (0,\infty)} \le
    c \| h \|_{s, \phi^{1 - s / \delta},(0,\infty)}, \, h \in {\mathfrak M}^+
    \end{equation}
    holds.
\end{corollary}
\begin{proof}
    By Theorem \ref{RT.thm.main.3}, inequality \eqref{RT.thm.main.3.eq.1} is
    equivalent to
    \begin{equation}\label{RT.cor.main.3.1.eq.4}
    \| T_{\Phi^2} f \|_{\beta, w, (0,\infty)} \le c \| f \|_{s, \phi,(0,\infty)}, \, f \in
    {\mathfrak M}^{\downarrow}.
    \end{equation}
    We know that \eqref{RT.cor.main.3.1.eq.4} is equivalent to
    \begin{equation}\label{RT.cor.main.3.1.eq.5}
    \| \widetilde{T} f \|_{\beta / \delta, w, (0,\infty)} \le c^{\delta} \| f \|_{s / \delta, \phi,(0,\infty)}, \, f \in
    {\mathfrak M}^{\downarrow},
    \end{equation}
    with
    $$
    \widetilde{T} (f) : = \left\{ T_{\Phi^2}(f^{1/ \delta})\right\}^{\delta},
    $$
    By Theorem \ref{Reduction.Theorem.thm.3.2}, we see that \eqref{RT.cor.main.3.1.eq.5} is equivalent to
    \begin{equation}\label{RT.cor.main.3.1.eq.6.0}
    \bigg\| \widetilde{T}\bigg( \frac{1}{\Phi^2(x)} \int_0^x h\Phi \bigg)\bigg\|_{\beta / \delta,w,(0,\infty)}
    \le c^{\delta} \| h \|_{s / \delta, \phi^{1 - s / \delta},(0,\infty)}, ~ h \in {\mathfrak M}^+(0,\infty).
    \end{equation}
    To complete the proof it suffices to note that \eqref{RT.cor.main.3.1.eq.6.0} is equivalent to \eqref{RT.cor.main.3.1.eq.2}.
\end{proof}

The following "dual" version of the reduction and equivalence
statements also hold true and may be proved analogously.
\begin{theorem}\label{RT.thm.main.4}
    Let $0 < \beta \le \infty$, $1 < s < \infty$, and let $T: {\mathfrak M}^+ \rightarrow {\mathfrak M}^+$
    satisfies conditions {\rm (i)-(iii)}.  Assume that $u,\,w \in {\mathcal W}(0,\infty)$
    and $v \in {\mathcal W}(0,\infty)$ be such that
    \begin{equation}\label{RT.thm.main.4.eq.0}
    \int_x^{\infty} v^{1-s^{\prime}}(t)\,dt < \infty \qquad \mbox{for all} \qquad x >
    0.
    \end{equation}
    Then inequality \eqref{RT.thm.main.4.eq.1} holds iff
    \begin{equation}\label{RT.thm.main.4.eq.2}
    \| T_{\Psi^2} f \|_{\beta, w, (0,\infty)} \le c \| f \|_{s, \psi,(0,\infty)}, \, f \in
    {\mathfrak M}^{\uparrow}
    \end{equation}
    holds, where
    $$
    \psi (x) \equiv \psi \big[v;s\big](x): = \bigg( \int_x^{\infty}
    v^{1-{s}^{\prime}}(t)\,dt\bigg)^{- \frac{s^{\prime}}{s^{\prime} + 1}}
    v^{1-{s}^{\prime}}(x)
    $$
    and
    $$
    \Psi(x) \equiv \Psi \big[v;s\big](x) : = \int_x^{\infty} \psi (t)\,dt = \bigg( \int_x^{\infty}
    v^{1-{s}^{\prime}}(t)\,dt\bigg)^{\frac{1}{s^{\prime} + 1}}
    $$
\end{theorem}

\begin{corollary}\label{RT.cor.main.4}
    Let $0 < \beta \le \infty$, $1 < s < \infty$, $0 < \delta \le s$, and let $T: {\mathfrak M}^+ \rightarrow {\mathfrak M}^+$
    satisfies conditions {\rm (i)-(iii)}.  Assume that $u,\,w \in {\mathcal W}(0,\infty)$
    and $v \in {\mathcal W}(0,\infty)$ be such that \eqref{RT.thm.main.4.eq.0} holds. Then
    inequality \eqref{RT.thm.main.4.eq.1}
    holds iff both
    \begin{equation}\label{RT.cor.main.4.eq.2}
    \bigg\| T_{\Psi^2}  \bigg( \left\{ \int_0^x h^{\delta}\right\}^{1 / \delta}\bigg) \bigg\|_{\beta, w, (0,\infty)} \le
    c \| h \|_{s, \Psi^{s / \delta} \psi^{1 - s / \delta},(0,\infty)}, \, h \in {\mathfrak M}^+,
    \end{equation}
    and
    \begin{equation}\label{RT.cor.main.4.eq.3}
    \| T_{\Psi^2} {\bf 1} \|_{\beta, w, (0,\infty)} \le c \| {\bf 1} \|_{s,
        \psi,(0,\infty)},
    \end{equation}
    hold.
\end{corollary}

\begin{corollary}\label{RT.cor.main.4.1}
    Let $0 < \beta \le \infty$, $1 < s < \infty$, $0 < \delta \le s$, and let $T: {\mathfrak M}^+ \rightarrow {\mathfrak M}^+$
    satisfies conditions {\rm (i)-(iii)}.  Assume that $u,\,w \in {\mathcal W}(0,\infty)$
    and $v \in {\mathcal W}(0,\infty)$ be such that \eqref{RT.thm.main.4.eq.0} holds. Then
    inequality \eqref{RT.thm.main.4.eq.1}
    holds iff
    \begin{equation}\label{RT.cor.main.4.1.eq.2}
    \bigg\| T_{\Psi^{2(1 - 1 / \delta)}} \bigg(\bigg\{\int_x^{\infty} h^{\delta}\Psi \bigg\}^{1  / \delta}\bigg) \bigg\|_{\beta, w, (0,\infty)} \le
    c \| h \|_{s, \psi^{1 - s / \delta},(0,\infty)}, \, h \in {\mathfrak M}^+
    \end{equation}
    holds.
\end{corollary}

The following theorem allows to reduce the iterated inequality
\eqref{RT.thm.main.3.eq.1} to the inequality on the cone of
non-decreasing functions.
\begin{theorem}\label{RT.thm.main.7}
    Let $0 < \beta \le \infty$, $1 < s < \infty$, and let $T: {\mathfrak M}^+ \rightarrow {\mathfrak M}^+$
    satisfies conditions {\rm (i)-(iii)}.  Assume that $u,\,w \in {\mathcal W}(0,\infty)$
    and $v \in {\mathcal W}(0,\infty)$ be such that \eqref{RT.thm.main.3.eq.0} holds. Then
    inequality \eqref{RT.thm.main.3.eq.1}
    holds iff both
    \begin{align*}\label{RT.cor.main.7.eq.2}
    \big\| T_{\Phi^2[v;s] \cdot \Psi^{2 / \delta} [\Phi[v;s]^{s / \delta} \phi[v;s]^{1-s / \delta};s / \delta]}  f
    \big\|_{\beta, w, (0,\infty)} \le c \| f \|_{s,    \psi [\Phi[v;s]^{s / \delta} \phi[v;s]^{1-s / \delta};s / \delta],(0,\infty)}, \, f \in
    {\mathfrak M}^{\uparrow},
    \end{align*}
    where $0 < \delta < s$,
    \begin{align*}
    \psi \big[\Phi[v;s]^{s / \delta} \phi[v;s]^{1-s / \delta};s / \delta\big] (x) & \\
    & \hspace{-3cm} \approx \bigg\{ \int_x^{\infty}\bigg( \int_0^t v^{1 - s'}\bigg)^{-\frac{s' + (s / \delta)'}{1 + s'}} v^{1-s'}(t)\,dt\bigg\}^{-\frac{(s / \delta)'}{1 + (s / \delta)'}} \bigg( \int_0^x v^{1 - s'}\bigg)^{-\frac{s' + (s / \delta)'}{1 + s'}}\!\!\! v^{1-s'}(x),\\
    \Psi \big[\Phi[v;s]^{s / \delta} \phi[v;s]^{1-s/\delta};s / \delta\big]  (x)
    & \approx \bigg\{ \int_x^{\infty}\bigg( \int_0^t v^{1 - s'}\bigg)^{-\frac{s' + (s / \delta)'}{1 + s'}} \!\!\! v^{1-s'}(t)\,dt \bigg\}^{\frac{1}{1 + (s / \delta)'}},
    \end{align*}
    and \eqref{RT.cor.main.3.eq.3} hold.
\end{theorem}
\begin{proof}
    By Corollary \ref{RT.cor.main.3}, \eqref{RT.thm.main.3.eq.1} holds
    iff both \eqref{RT.cor.main.3.eq.2} and
    \eqref{RT.cor.main.3.eq.3} hold. It is easy to see that \eqref{RT.cor.main.3.eq.2} is equivalent to
    \begin{equation}\label{RT.thm.main.7.eq.1}
    \bigg\|\bigg[ T_{\Phi^2}  \bigg( \bigg\{ \int_x^{\infty} h\bigg\}^{1 / \delta}\bigg) \bigg]^{\delta}\bigg\|_{\beta / \delta, w, (0,\infty)} \le
    c^{\delta} \| h \|_{s / \delta, \Phi^{s / \delta} \phi^{1 - s / \delta},(0,\infty)}, \, h \in {\mathfrak M}^+,
    \end{equation}
    Since
    \begin{align*}
    \psi \big[\Phi[v;s]^{s / \delta} \phi[v;s]^{1-s / \delta};s / \delta\big] (x) & \\
    & \hspace{-3cm} = \bigg(\int_x^{\infty} \Phi[v;s]^{-(s/\delta)'} \phi[v;s] \bigg)^{- \frac{(s / \delta)'}{(s / \delta)' + 1}}
    \Phi[v;s]^{-(s/\delta)'}(x) \phi[v;s](x)   \\
    & \hspace{-3cm} \approx \bigg( \Phi[v;s]^{1-(s / \delta)'}(x) -
    \Phi[v;s]^{1-(s / \delta)'}(\infty)\bigg)^{- \frac{(s / \delta)'}{(s / \delta)' + 1}} \Phi[v;s]^{-(s / \delta)'}(x) \phi[v;s](x)\\
    & \hspace{-3cm} \approx \bigg\{ \bigg( \int_0^x v^{1 - s'}\bigg)^{\frac{1 - (s / \delta)'}{1 + s'}} \!\!\!\!\!\!- \bigg( \int_0^{\infty} v^{1 - s'}\bigg)^{\frac{1 - (s / \delta)'}{1 + s'}}\bigg\}^{-\frac{(s / \delta)'}{1 + (s / \delta)'}} \bigg( \int_0^x v^{1 - s'}\bigg)^{-\frac{s' + (s / \delta)'}{1 + s'}} \!\!\! v^{1-s'}(x) \\
    & \hspace{-3cm} \approx \bigg\{ \int_x^{\infty}\bigg( \int_0^t v^{1 - s'}\bigg)^{-\frac{s' + (s / \delta)'}{1 + s'}} v^{1-s'}(t)\,dt\bigg\}^{-\frac{(s / \delta)'}{1 + (s / \delta)'}} \bigg( \int_0^x v^{1 - s'}\bigg)^{-\frac{s' + (s / \delta)'}{1 + s'}} \!\!\! v^{1-s'}(x),
    \end{align*}
    and
    \begin{align*}
    \Psi \big[\Phi[v;s]^{s / \delta} \phi[v;s]^{1- s / \delta};s / \delta\big] (x) & \\
    & \hspace{-3cm} = \bigg(\int_x^{\infty} \Phi[v;s]^{-(s / \delta)'} \phi[v;s]\bigg)^{\frac{1}{(s / \delta)' + 1}} \\
    & \hspace{-3cm} \approx \bigg( \Phi[v;s]^{1-(s / \delta)'}(x) - \Phi[v;s]^{1-(s / \delta)'}(\infty)\bigg)^{\frac{1}{(s / \delta)' + 1}} \\
    & \hspace{-3cm} \approx \bigg\{ \bigg( \int_0^x v^{1 - s'}\bigg)^{\frac{1 - (s / \delta)'}{1 + s'}} - \bigg( \int_0^{\infty} v^{1 - s'}\bigg)^{\frac{1 - (s / \delta)'}{1 + s'}}\bigg\}^{\frac{1}{1 + (s / \delta)'}} \\
    & \hspace{-3cm} \approx \bigg\{ \int_x^{\infty}\bigg( \int_0^t v^{1 - s'}\bigg)^{-\frac{s' + (s / \delta)'}{1 + s'}} v^{1-s'}(t)\,dt \bigg\}^{\frac{1}{1 + (s / \delta)'}},
    \end{align*}
    by Theorem \ref{RT.thm.main.4}, we complete the proof.
\end{proof}

\begin{corollary}\label{RT.thm.main.7.cor.1}
    Let $0 < \beta \le \infty$, $1 < s < \infty$, and let $T: {\mathfrak M}^+ \rightarrow {\mathfrak M}^+$
    satisfies conditions {\rm (i)-(iii)}.  Assume that $u,\,w \in {\mathcal W}(0,\infty)$
    and $v \in {\mathcal W}(0,\infty)$ be such that \eqref{RT.thm.main.3.eq.0} holds. Then
    inequality \eqref{RT.thm.main.3.eq.1}
    holds iff both
    \begin{align*}\label{RT.thm.main.7.cor.1.eq.2}
    \big\| T_{\Phi^2[v;s] \cdot \Psi^{4 / s} [\Phi[v;s]^2 \phi[v;s]^{-1};2]}  f
    \big\|_{\beta, w, (0,\infty)} \le c \| f \|_{s,    \psi [\Phi[v;s]^2 \phi[v;s]^{-1};2],(0,\infty)}, \, f \in
    {\mathfrak M}^{\uparrow},
    \end{align*}
    where
    \begin{align*}
    \psi \big[\Phi[v;s]^2 \phi[v;s]^{-1};2\big] (x) & \\
    &\hspace{-3cm} \approx \bigg\{ \int_x^{\infty} \bigg( \int_0^t v^{1 - s'}\bigg)^{-\frac{2 + s'}{1 + s'}} v^{1-s'}(t)\,dt\bigg\}^{-\frac{2}{3}} \bigg( \int_0^x v^{1 - s'}\bigg)^{-\frac{2 + s'}{1 + s'}} v^{1-s'}(x),\\
    \Psi \big[\Phi[v;s]^{2} \phi[v;s]^{-1};2\big]  (x)
    & \approx \bigg\{ \int_x^{\infty} \bigg( \int_0^t v^{1 - s'}\bigg)^{-\frac{2 + s'}{1 + s'}} v^{1-s'}(t)\,dt\bigg\}^{\frac{1}{3}},
    \end{align*}
    and \eqref{RT.cor.main.3.eq.3} hold.
\end{corollary}
\begin{proof}
    The statement follows by Theorem \ref{RT.thm.main.7} with $\delta = s /2$.
\end{proof}

The following "dual" statement also holds true and may be proved
analogously.
\begin{theorem}\label{RT.thm.main.8}
    Let $0 < \beta \le \infty$, $1 < s < \infty$, and let $T: {\mathfrak M}^+ \rightarrow {\mathfrak M}^+$
    satisfies conditions {\rm (i)-(iii)}.  Assume that $u,\,w \in {\mathcal W}(0,\infty)$
    and $v \in {\mathcal W}(0,\infty)$ be such that \eqref{RT.thm.main.4.eq.0} holds. Then
    inequality  \eqref{RT.thm.main.4.eq.1} holds iff both
    \begin{align*}
    \big\| T_{\Psi^2[v;s] \cdot \Phi^{2 / \delta}[\Psi[v;s]^{s / \delta} \psi[v;s]^{1-s / \delta};s / \delta]} f
    \big\|_{\beta, w, (0,\infty)} \le c \| f \|_{s,   \phi[\Psi[v;s]^{s / \delta} \psi[v;s]^{1-s / \delta};s / \delta],(0,\infty)}, \, f \in {\mathfrak M}^{\downarrow},
    \end{align*}
    where $0 < \delta <s$,
    \begin{align*}
    \phi \big[\Psi[v;s]^{s / \delta} \psi[v;s]^{1-s / \delta};s  / \delta\big] (x) & \\
    & \hspace{-3cm} \approx \bigg\{ \int_0^x \bigg( \int_t^{\infty} v^{1 - s'}\bigg)^{-\frac{s' + (s / \delta)'}{1 + s'}} \!\!\! v^{1-s'}(t)\,dt \bigg\}^{-\frac{(s / \delta)'}{1 + (s / \delta)'}} \bigg( \int_x^{\infty} v^{1 - s'}\bigg)^{-\frac{s' + (s / \delta)'}{1 + s'}} \!\!\! v^{1-s'}(x),\\
    \Phi \big[\Psi[v;s]^{s / \delta} \psi[v;s]^{1-s/\delta};s / \delta\big]  (x)
    & \approx \bigg\{\int_0^x \bigg( \int_t^{\infty} v^{1 - s'}\bigg)^{-\frac{s' + (s / \delta)'}{1 + s'}} \!\!\! v^{1-s'}(t)\,dt\bigg\}^{\frac{1}{1 + (s / \delta)'}},
    \end{align*}
    and \eqref{RT.cor.main.3.eq.3} hold.
\end{theorem}

\begin{corollary}\label{RT.thm.main.8.cor.1}
    Let $0 < \beta \le \infty$, $1 < s < \infty$, and let $T: {\mathfrak M}^+ \rightarrow {\mathfrak M}^+$
    satisfies conditions {\rm (i)-(iii)}.  Assume that $u,\,w \in {\mathcal W}(0,\infty)$
    and $v \in {\mathcal W}(0,\infty)$ be such that \eqref{RT.thm.main.4.eq.0} holds. Then
    inequality  \eqref{RT.thm.main.4.eq.1} holds iff both
    \begin{align*}
    \big\| T_{\Psi^2[v;s] \cdot \Phi^{4 / s}[\Psi[v;s]^2 \psi[v;s]^{-1};2]} f
    \big\|_{\beta, w, (0,\infty)} \le c \| f \|_{s,   \phi[\Psi[v;s]^2 \psi[v;s]^{-1};2],(0,\infty)}, \, f \in {\mathfrak M}^{\downarrow},
    \end{align*}
    where
    \begin{align*}
    \phi \big[\Psi[v;s]^2 \psi[v;s]^{-1};2\big] (x) & \\
    & \hspace{-3cm} \approx \bigg\{ \int_0^x \bigg( \int_t^{\infty} v^{1 - s'}\bigg)^{-\frac{2 + s'}{1 + s'}} v^{1-s'}(t)\,dt\bigg\}^{-\frac{2}{3}} \bigg( \int_x^{\infty} v^{1 - s'}\bigg)^{-\frac{2 + s'}{1 + s'}} v^{1-s'}(x),\\
    \Phi \big[\Psi[v;s]^2 \psi[v;s]^{-1};2\big]  (x) & \approx \bigg\{\int_0^x \bigg( \int_t^{\infty} v^{1 - s'}\bigg)^{-\frac{2 + s'}{1 + s'}} v^{1-s'}(t)\,dt\bigg\}^{\frac{1}{3}},
    \end{align*}
    and \eqref{RT.cor.main.3.eq.3} hold.
\end{corollary}

\subsection{The case $s = 1$}\label{RTs}

In this case we have the following results.

\begin{theorem}\label{RT.thm.main.8.0}
    Let $0 < \beta \le \infty$, and let $T: {\mathfrak M}^+ \rightarrow {\mathfrak M}^+$ satisfies conditions {\rm (i)-(iii)}.  Assume that $u,\,w \in {\mathcal W}(0,\infty)$
    and $v \in {\mathcal W}(0,\infty)$ be such that $V(x) < \infty$ for all $x > 0$.
    Then inequality
    \begin{equation}    \label{RT.thm.main.8.0.eq.1}
    \left\|T \bigg(\int_0^x h\bigg)\right\|_{\beta,w,(0,\infty)} \leq c
    \,\|h\|_{1,V^{-1},(0,\infty)}, \, h \in {\mathfrak M}^+,
    \end{equation}
    holds iff
    \begin{equation}\label{RT.thm.main.8.0.eq.2}
    \| T_{V^2} f \|_{\beta, w, (0,\infty)} \le c \| f  \|_{1, v,(0,\infty)}, \, f \in
    {\mathfrak M}^{\downarrow}.
    \end{equation}
\end{theorem}

\begin{proof}
    Inequality
    \eqref{RT.thm.main.8.0.eq.1} is equivalent to the inequality
    \begin{equation}\label{RT.thm.main.8.0.eq.3}
    \left\|T_{V^2} \bigg(\frac{1}{V^2(x)}\int_0^x hV\bigg)\right\|_{\beta,w,(0,\infty)} \leq c
    \,\|h\|_{1,(0,\infty)}, \, h \in {\mathfrak M}^+.
    \end{equation}
    By Theorem \ref{Reduction.Theorem.thm.3.2}, inequality \eqref{RT.thm.main.8.0.eq.3}  is equivalent to \eqref{RT.thm.main.8.0.eq.2}.
\end{proof}

\begin{corollary}\label{RT.thm.main.cor.8.0}
    Let $0 < \beta \le \infty$, $0< \delta \le 1$, and let $T: {\mathfrak M}^+ \rightarrow {\mathfrak M}^+$ satisfies conditions {\rm (i)-(iii)}.  Assume that $u,\,w \in {\mathcal W}(0,\infty)$
    and $v \in {\mathcal W}(0,\infty)$ be such that $V(x) < \infty$ for all $x > 0$.
    Then inequality \eqref{RT.thm.main.8.0.eq.1}
    holds iff both
    \begin{equation}\label{RT.thm.main.cor.8.0.eq.2}
    \bigg\| T_{V^2}  \bigg( \left\{ \int_x^{\infty} h^{\delta}\right\}^{1 / \delta}\bigg) \bigg\|_{\beta, w, (0,\infty)} \le c \| h \|_{1, V^{1 / \delta} v^{1 - 1 / \delta},(0,\infty)}, \, h \in {\mathfrak M}^+,
    \end{equation}
    and
    \begin{equation}\label{RT.thm.main.cor.8.0.eq.3}
    \| T_{V^2} ({\bf 1})\|_{\beta, w, (0,\infty)} \le c \| {\bf 1} \|_{1,v,(0,\infty)},
    \end{equation}
    hold.
\end{corollary}

\begin{proof}
    By Theorem \ref{RT.thm.main.8.0}, inequality \eqref{RT.thm.main.8.0.eq.1} is
    equivalent to \eqref{RT.thm.main.8.0.eq.2}.
    Since \eqref{RT.thm.main.8.0.eq.2} is equivalent to
    \begin{equation}\label{RT.thm.main.cor.8.0.eq.5}
    \left\| \left\{ T_{V^2}(f^{1/ \delta})\right\}^{\delta} \right\|_{\beta / \delta, w, (0,\infty)} \le c^{\delta} \| f \|_{1 / \delta, v,(0,\infty)}, \, f \in
    {\mathfrak M}^{\downarrow},
    \end{equation}
    it remains to apply Theorem \ref{Reduction.Theorem.thm.3.1}.
\end{proof}

\begin{corollary}\label{RT.thm.main.cor.8.0.1}
    Let $0 < \beta \le \infty$, $0< \delta \le 1$, and let $T: {\mathfrak M}^+ \rightarrow {\mathfrak M}^+$ satisfies conditions {\rm (i)-(iii)}.  Assume that $u,\,w \in {\mathcal W}(0,\infty)$
    and $v \in {\mathcal W}(0,\infty)$ be such that $V(x) < \infty$ for all $x > 0$.
    Then inequality \eqref{RT.thm.main.8.0.eq.1}
    holds iff
    \begin{equation}\label{RT.thm.main.cor.8.0.1.eq.2}
    \bigg\| T_{V^{2(1 - 1 / \delta)}} \bigg(\bigg\{\int_0^x h^{\delta} V \bigg\}^{1  / \delta}\bigg) \bigg\|_{\beta, w, (0,\infty)} \le
    c \| h \|_{1, v^{1 - 1 / \delta},(0,\infty)}, \, h \in {\mathfrak M}^+
    \end{equation}
    holds.
\end{corollary}

\begin{proof}
    By Theorem \ref{RT.thm.main.8.0}, inequality \eqref{RT.thm.main.8.0.eq.1} is
    equivalent to \eqref{RT.thm.main.cor.8.0.eq.5}.
    By Theorem \ref{Reduction.Theorem.thm.3.2}, we see that \eqref{RT.thm.main.cor.8.0.eq.5} is equivalent to
    \begin{equation}\label{RT.cor.main.3.1.eq.6}
    \bigg\| \bigg\{T_{V^2}\bigg( \bigg[\frac{1}{V^2(x)} \int_0^x hV \bigg]^{1 / \delta}\bigg)\bigg\}^{\delta}\bigg\|_{\beta / \delta,w,(0,\infty)}
    \le c^{\delta} \| h \|_{1 / \delta, v^{1 - 1 / \delta},(0,\infty)}, ~ h \in {\mathfrak M}^+(0,\infty).
    \end{equation}
    To complete the proof it suffices to note that \eqref{RT.cor.main.3.1.eq.6} is equivalent to \eqref{RT.thm.main.cor.8.0.1.eq.2}.
\end{proof}

The following theorem allows to reduce the iterated inequality
\eqref{RT.thm.main.8.0.eq.1} to the inequality on the cone of
non-decreasing functions.
\begin{theorem}\label{RT.thm.main.17.0}
    Let $0 < \beta \le \infty$, and let $T: {\mathfrak M}^+ \rightarrow {\mathfrak M}^+$ satisfies conditions {\rm (i)-(iii)}.  Assume that $u,\,w \in {\mathcal W}(0,\infty)$
    and $v \in {\mathcal W}(0,\infty)$ be such that $V(x) < \infty$ for all $x > 0$.
    Then inequality \eqref{RT.thm.main.8.0.eq.1}
    holds iff both
    \begin{equation}\label{RT.thm.main.17.0.eq.2}
    \big\| T_{V^2 \cdot \Psi^{2 / \delta} [V^{1 / \delta} v^{1-1 / \delta};1 / \delta]}  f
    \big\|_{\beta, w, (0,\infty)} \le c \| f \|_{1,    \psi [V^{1 / \delta} v^{1 - 1 / \delta};1 / \delta],(0,\infty)}, \, f \in
    {\mathfrak M}^{\uparrow},
    \end{equation}
    where $0< \delta < 1$,
    \begin{align*}
    \psi [V^{1/\delta} v^{1 - 1 / \delta};1 / \delta](x) & \approx \bigg(\int_x^{\infty} V^{-(1/\delta)'}v\bigg)^{- \frac{(1/\delta)'}{1 + (1/\delta)'}} V^{-(1/\delta)'}(x)v(x), \\
    \Psi [V^{1/\delta} v^{1 - (1/\delta)};1/\delta](x) & \approx \bigg(\int_x^{\infty} V^{-(1/\delta)'}v\bigg)^{\frac{1}{1 + (1/\delta)'}},
    \end{align*}
    and \eqref{RT.thm.main.cor.8.0.eq.3} hold.
\end{theorem}

\begin{proof}
    By Corollary \ref{RT.thm.main.cor.8.0}, inequality
    \eqref{RT.thm.main.8.0.eq.1} holds iff both \eqref{RT.thm.main.cor.8.0.eq.2}
    and \eqref{RT.thm.main.cor.8.0.eq.3} hold. It is easy to see that \eqref{RT.thm.main.cor.8.0.eq.2} is equivalent to
    \begin{equation}\label{RT.thm.main.17.0.eq.4}
    \bigg\|\bigg[ T_{V^2}  \bigg( \left\{ \int_x^{\infty} h\right\}^{1 / \delta}\bigg)\bigg]^{\delta} \bigg\|_{\beta / \delta, w, (0,\infty)} \le c^{\delta} \| h \|_{1/\delta, V^{1/\delta} v^{1 - 1 / \delta},(0,\infty)}, \, h \in {\mathfrak M}^+.
    \end{equation}
    By Theorem \ref{RT.thm.main.4}, inequality \eqref{RT.thm.main.17.0.eq.4}  is equivalent to
    \begin{align*}
    \bigg\|\bigg[ T_{V^2 \Psi^{2 / \delta}\big[V^{1 / \delta}v^{1 - 1/ \delta};1 / \delta\big]} \big(  f^{1/\delta} \big)\bigg]^{\delta} \bigg\|_{\beta / \delta, w, (0,\infty)}  \le c^{\delta} \| f \|_{1 / \delta, \psi [V^{1/\delta} v^{1 - 1 / \delta};1  / \delta],(0,\infty)}, \, f \in {\mathfrak M}^{\uparrow}, \label{RT.thm.main.17.0.eq.5}
    \end{align*}
    which is evidently equivalent to \eqref{RT.thm.main.17.0.eq.2}.

    It remains to note that
    \begin{align*}
    \psi [V^{1/\delta} v^{1 - 1 / \delta};1 / \delta](x) & \approx \bigg(\int_x^{\infty} V^{-(1/\delta)'}v\bigg)^{- \frac{(1/\delta)'}{1 + (1/\delta)'}} V^{-(1/\delta)'}(x)v(x), \\
    \Psi [V^{1/\delta} v^{1 - (1/\delta)};1/\delta](x) & \approx \bigg(\int_x^{\infty} V^{-(1/\delta)'}v\bigg)^{\frac{1}{1 + (1/\delta)'}}.
    \end{align*}

\end{proof}

\begin{corollary}\label{RT.thm.main.17.0.cor.1}
    Let $0 < \beta \le \infty$, and let $T: {\mathfrak M}^+ \rightarrow {\mathfrak M}^+$ satisfies conditions {\rm (i)-(iii)}.  Assume that $u,\,w \in {\mathcal W}(0,\infty)$
    and $v \in {\mathcal W}(0,\infty)$ be such that $V(x) < \infty$ for all $x > 0$.
    Then inequality \eqref{RT.thm.main.8.0.eq.1}
    holds iff both
    \begin{equation}\label{RT.thm.main.17.0.cor.1.eq.2}
    \bigg\| T_{V^2 \Psi^4\big[V^2v^{-1};2\big]} ( f) \bigg\|_{\beta, w, (0,\infty)} \le c \| f \|_{1, \psi [V^2 v^{- 1};2],(0,\infty)}, \, f \in {\mathfrak M}^{\uparrow},
    \end{equation}
    where
    \begin{align*}
    \psi [V^2 v^{- 1};2](x) & \approx \bigg(\int_x^{\infty} V^{-2}v\bigg)^{- 2 / 3} V^{-2}(x)v(x) \\
    \Psi [V^2 v^{- 1};2](x) & \approx \bigg(\int_x^{\infty} V^{-2}v\bigg)^{1 / 3},
    \end{align*}
    and \eqref{RT.thm.main.cor.8.0.eq.3} hold.
\end{corollary}

\begin{proof}
    The statement follows by Theorem \ref{RT.thm.main.17.0} with $\delta = 1 /2$.
\end{proof}

The following statement immediately follows from Theorem
\ref{RT.thm.main.8.0}.
\begin{corollary}\label{RT.thm.main.3.0}
    Let $0 < \beta \le \infty$, and let $T: {\mathfrak M}^+ \rightarrow {\mathfrak M}^+$ satisfies conditions {\rm (i)-(iii)}.  Assume that $u,\,w \in {\mathcal W}(0,\infty)$
    and $v \in {\mathcal W}(0,\infty)$ be such that $V_*(x) < \infty$ for all $x > 0$ and $V_*(0) = \infty$.
    Then inequality
    \begin{equation}
    \label{RT.thm.main.3.0.eq.1}
    \left\|T \bigg(\int_0^x h\bigg)\right\|_{\beta,w,(0,\infty)} \leq c
    \,\|h\|_{1,V_*,(0,\infty)}, \, h \in {\mathfrak M}^+,
    \end{equation}
    holds iff
    \begin{equation}\label{RT.thm.main.3.0.eq.2}
    \| T_{V_*^{-2}} f \|_{\beta, w, (0,\infty)} \le c \| f  \|_{1, v / V_*^2,(0,\infty)}, \, f \in
    {\mathfrak M}^{\downarrow}
    \end{equation}
    holds.
\end{corollary}

\begin{proof}
    Since
    $$
    V_* (x) = \left( \int_0^x \frac{v}{V_*^2}\right)^{-1}, ~ x > 0,
    $$
    it remains to apply Theorem \ref{RT.thm.main.8.0}.
\end{proof}

\begin{corollary}\label{RT.cor.main.3.0.1}
    Let $0 < \beta \le \infty$, $0 < \delta \le 1$, and let $T: {\mathfrak M}^+ \rightarrow {\mathfrak M}^+$
    satisfies conditions {\rm (i)-(iii)}.  Assume that $u,\,w \in {\mathcal W}(0,\infty)$
    and $v \in {\mathcal W}(0,\infty)$ be such that $V_*(x) < \infty$ for all $x > 0$ and $V_*(0) = \infty$. Then
    inequality \eqref{RT.thm.main.3.0.eq.1}
    holds iff both
    \begin{equation}\label{RT.cor.main.3.0.1.eq.1}
    \bigg\| T_{V_*^{-2}}  \bigg( \left\{ \int_x^{\infty} h^{\delta}\right\}^{1 / \delta}\bigg) \bigg\|_{\beta, w, (0,\infty)} \le
    c \| h \|_{1, V_*^{1 / \delta - 2} v^{1 - 1 / \delta},(0,\infty)}, \, h \in {\mathfrak M}^+,
    \end{equation}
    holds.
\end{corollary}

\begin{corollary}\label{RT.cor.main.3.1.1}
    Let $0 < \beta \le \infty$, $0 < \delta \le 1$, and let $T: {\mathfrak M}^+ \rightarrow {\mathfrak M}^+$
    satisfies conditions {\rm (i)-(iii)}.  Assume that $u,\,w \in {\mathcal W}(0,\infty)$
    and $v \in {\mathcal W}(0,\infty)$ be such that $V_*(x) < \infty$ for all $x > 0$ and $V_*(0) = \infty$. Then
    inequality \eqref{RT.thm.main.3.0.eq.1}
    holds iff
    \begin{equation}\label{RT.cor.main.3.1.1.eq.1}
    \bigg\| T_{V_*^{2(1 / \delta - 1)}} \bigg(\bigg\{\int_0^x h^{\delta} \bigg\}^{1  / \delta}\bigg) \bigg\|_{\beta, w, (0,\infty)} \le
    c \| h \|_{1, V_*^{3 / \delta - 2}v^{1 - 1 / \delta},(0,\infty)}, \, h \in {\mathfrak M}^+
    \end{equation}
    holds.
\end{corollary}

The following "dual" statements also hold true and may be proved
analogously.
\begin{theorem}\label{RT.thm.main.9.0}
    Let $0 < \beta \le \infty$, and let $T: {\mathfrak M}^+ \rightarrow {\mathfrak M}^+$ satisfies conditions {\rm (i)-(iii)}.  Assume that $u,\,w \in {\mathcal W}(0,\infty)$
    and $v \in {\mathcal W}(0,\infty)$ be such that $V_*(x) < \infty$ for all $x > 0$.
    Then inequality
    \begin{equation}\label{RT.thm.main.9.0.eq.1}
    \left\|T \bigg(\int_x^{\infty} h\bigg)\right\|_{\beta,w,(0,\infty)} \leq c
    \,\|h\|_{1,V_*^{-1},(0,\infty)}, \, h \in {\mathfrak M}^+,
    \end{equation}
    holds iff
    \begin{equation}\label{RT.thm.main.9.0.eq.2}
    \| T_{V_*^2} f \|_{\beta, w, (0,\infty)} \le c \| f  \|_{1, v,(0,\infty)}, \, f \in
    {\mathfrak M}^{\uparrow}.
    \end{equation}
\end{theorem}

\begin{corollary}\label{RT.thm.main.cor.9.0}
    Let $0 < \beta \le \infty$, $0< \delta \le 1$, and let $T: {\mathfrak M}^+ \rightarrow {\mathfrak M}^+$ satisfies conditions {\rm (i)-(iii)}.  Assume that $u,\,w \in {\mathcal W}(0,\infty)$
    and $v \in {\mathcal W}(0,\infty)$ be such that $V_*(x) < \infty$ for all $x > 0$.
    Then inequality \eqref{RT.thm.main.9.0.eq.1}
    holds iff both
    \begin{equation}\label{RT.thm.main.cor.9.0.eq.2}
    \bigg\| T_{V_*^2}  \bigg( \left\{ \int_0^x h^{\delta}\right\}^{1 / \delta}\bigg) \bigg\|_{\beta, w, (0,\infty)} \le c \| h \|_{1, V_*^{1 / \delta} v^{1 - 1 / \delta},(0,\infty)}, \, h \in {\mathfrak M}^+,
    \end{equation}
    and
    \begin{equation}\label{RT.thm.main.cor.9.0.eq.3}
    \| T_{V_*^2} ({\bf 1})\|_{\beta, w, (0,\infty)} \le c \| {\bf 1} \|_{1,v,(0,\infty)},
    \end{equation}
    hold.
\end{corollary}

\begin{corollary}\label{RT.thm.main.cor.9.0.1}
    Let $0 < \beta \le \infty$, $0< \delta \le 1$, and let $T: {\mathfrak M}^+ \rightarrow {\mathfrak M}^+$ satisfies conditions {\rm (i)-(iii)}.  Assume that $u,\,w \in {\mathcal W}(0,\infty)$
    and $v \in {\mathcal W}(0,\infty)$ be such that $V_*(x) < \infty$ for all $x > 0$.
    Then inequality \eqref{RT.thm.main.9.0.eq.1}
    holds iff
    \begin{equation}\label{RT.thm.main.cor.9.0.1.eq.2}
    \bigg\| T_{V_*^{2(1 - 1 / \delta)}} \bigg(\bigg\{\int_x^{\infty} h^{\delta} V_* \bigg\}^{1  / \delta}\bigg) \bigg\|_{\beta, w, (0,\infty)} \le
    c \| h \|_{1, v^{1 - 1 / \delta},(0,\infty)}, \, h \in {\mathfrak M}^+
    \end{equation}
    holds.
\end{corollary}

\begin{theorem}\label{RT.thm.main.18.0}
    Let $0 < \beta \le \infty$, and let $T: {\mathfrak M}^+ \rightarrow {\mathfrak M}^+$ satisfies conditions {\rm (i)-(iii)}.  Assume that $u,\,w \in {\mathcal W}(0,\infty)$
    and $v \in {\mathcal W}(0,\infty)$ be such that $V_*(x) < \infty$ for all $x > 0$.
    Then inequality \eqref{RT.thm.main.9.0.eq.1}
    holds iff both
    \begin{equation}\label{RT.thm.main.18.0.eq.2}
    \big\| T_{V^2 \cdot \Phi^{2 / \delta} [V_*^{1 / \delta} v^{1-1 / \delta};1 / \delta]}  f
    \big\|_{\beta, w, (0,\infty)} \le c \| f \|_{1, \phi [V_*^{1 / \delta} v^{1 - 1 / \delta};1 / \delta],(0,\infty)}, \, f \in
    {\mathfrak M}^{\downarrow},
    \end{equation}
    where $0< \delta < 1$,
    \begin{align*}
    \phi [V_*^{1/\delta} v^{1 - 1 / \delta};1 / \delta](x) & \approx \bigg(\int_0^x V_*^{-(1/\delta)'}v\bigg)^{- \frac{(1/\delta)'}{1 + (1/\delta)'}} V_*^{-(1/\delta)'}(x)v(x), \\
    \Phi [V_*^{1/\delta} v^{1 - (1/\delta)};1/\delta](x) & \approx \bigg(\int_0^x V_*^{-(1/\delta)'}v\bigg)^{\frac{1}{1 + (1/\delta)'}},
    \end{align*}
    and \eqref{RT.thm.main.cor.9.0.eq.3} hold.
\end{theorem}

\begin{corollary}\label{RT.cor.main.18.0}
    Let $0 < \beta \le \infty$, and let $T: {\mathfrak M}^+ \rightarrow {\mathfrak M}^+$ satisfies conditions {\rm (i)-(iii)}.  Assume that $u,\,w \in {\mathcal W}(0,\infty)$
    and $v \in {\mathcal W}(0,\infty)$ be such that $V_*(x) < \infty$ for all $x > 0$.
    Then inequality \eqref{RT.thm.main.9.0.eq.1}
    holds iff both
    \begin{equation*}
    \bigg\| T_{ V_*^2 \Phi^4\big[V_*^2v^{-1};2\big]} (f) \bigg\|_{\beta, w, (0,\infty)} \le c \| f \|_{1, \phi [V_*^2 v^{- 1};2],(0,\infty)}, \, f \in {\mathfrak M}^{\downarrow},
    \end{equation*}
    where
    \begin{align*}
    \phi [V_*^2 v^{- 1};2](x) & \approx \bigg(\int_0^x V_*^{-2}v\bigg)^{- 2 / 3} V_*^{-2}(x)v(x) \\
    \Phi [V_*^2 v^{- 1};2](x) & \approx \bigg(\int_0^x V_*^{-2}v\bigg)^{1 / 3},
    \end{align*}
    and \eqref{RT.thm.main.cor.9.0.eq.3} hold.
\end{corollary}

\begin{corollary}\label{RT.thm.main.4.0}
    Let $0 < \beta < \infty$, and let $T: {\mathfrak M}^+ \rightarrow {\mathfrak M}^+$ satisfies conditions {\rm (i)-(iii)}.  Assume that $u,\,w \in {\mathcal W}(0,\infty)$
    and $v \in {\mathcal W}(0,\infty)$ be such that $V(x) < \infty$ for all $x > 0$ and $V(\infty) = \infty$.
    Then inequality
    \begin{equation}
    \label{RT.thm.main.4.0.eq.1}
    \left\|T \bigg(\int_x^{\infty} h\bigg)\right\|_{\beta,w,(0,\infty)} \leq c
    \,\|h\|_{1,V,(0,\infty)}, \, h \in {\mathfrak M}^+,
    \end{equation}
    holds iff
    \begin{equation}\label{RT.thm.main.4.0.eq.2}
    \| T_{V^{-2}} f \|_{\beta, w, (0,\infty)} \le c \| f  \|_{1, v / V^2,(0,\infty)}, \, f \in
    {\mathfrak M}^{\downarrow}.
    \end{equation}
\end{corollary}

\begin{corollary}\label{RT.cor.main.4.0.1}
    Let $0 < \beta \le \infty$, $0 < \delta \le 1$, and let $T: {\mathfrak M}^+ \rightarrow {\mathfrak M}^+$
    satisfies conditions {\rm (i)-(iii)}.  Assume that $u,\,w \in {\mathcal W}(0,\infty)$
    and $v \in {\mathcal W}(0,\infty)$ be such that $V(x) < \infty$ for all $x > 0$ and $V(\infty) = \infty$. Then
    inequality \eqref{RT.thm.main.4.0.eq.1}
    holds iff both
    \begin{equation}\label{RT.cor.main.4.0.1.eq.1}
    \bigg\| T_{V^{-2}}  \bigg( \left\{ \int_0^x h^{\delta}\right\}^{1 / \delta}\bigg) \bigg\|_{\beta, w, (0,\infty)} \le
    c \| h \|_{1, V^{1 / \delta - 2} v^{1 - 1 / \delta},(0,\infty)}, \, h \in {\mathfrak M}^+,
    \end{equation}
    holds.
\end{corollary}

\begin{corollary}\label{RT.cor.main.4.1.1}
    Let $0 < \beta \le \infty$, $0 < \delta \le 1$, and let $T: {\mathfrak M}^+ \rightarrow {\mathfrak M}^+$
    satisfies conditions {\rm (i)-(iii)}.  Assume that $u,\,w \in {\mathcal W}(0,\infty)$
    and $v \in {\mathcal W}(0,\infty)$ be such that $V(x) < \infty$ for all $x > 0$ and $V(\infty) = \infty$. Then
    inequality \eqref{RT.thm.main.4.0.eq.1}
    holds iff
    \begin{equation}\label{RT.cor.main.4.1.1.eq.1}
    \bigg\| T_{V^{2(1 / \delta - 1)}} \bigg(\bigg\{\int_x^{\infty} h^{\delta} \bigg\}^{1  / \delta}\bigg) \bigg\|_{\beta, w, (0,\infty)} \le
    c \| h \|_{1, V^{3 / \delta - 2}v^{1 - 1 / \delta},(0,\infty)}, \, h \in {\mathfrak M}^+
    \end{equation}
    holds.
\end{corollary}

\section{Equivalence theorems for the weighted inequalities on the cones of monotone functions}\label{WeightedIneqOnTheCone}

As it is mentioned in the introduction, by substitution of variables
it is possible to change the cone of non-decreasing functions to the
cone of non-increasing functions and vice versa, when considering
inequalities \eqref{Reduction.Theorem.eq.1.1} and
\eqref{Reduction.Theorem.eq.1.1.00} for integral operators $T$. But
this procedure changes $T$ also as usually to the "dual" operator.

The following theorems allows to change the cones to each other not
changing the operator $T$.
\begin{theorem}\label{RT.thm.main.9}
    Let $0 < \beta \le \infty$, $0  < s < \infty$, and let $T: {\mathfrak M}^+ \rightarrow {\mathfrak M}^+$
    satisfies conditions {\rm (i)-(iii)}. Assume that $u,\,w \in {\mathcal W}(0,\infty)$
    and $v \in {\mathcal W}(0,\infty)$ be such that $V(x) < \infty$ for all $x > 0$ holds. Then
    inequality \eqref{Reduction.Theorem.eq.1.1} holds if and only if
    both
    \begin{equation}\label{RT.thm.main.9.eq.1}
    \left\|  T_{\{\Psi [V^{s / \delta}v^{1 - s / \delta};s/\delta]\}^{2 / \delta}} (f) \right\|_{\beta, w, (0,\infty)} \le c \|f\|_{s,\psi[V^{s / \delta}v^{1 - s / \delta};s/\delta],(0,\infty)}, \, f \in {\mathfrak M}^{\uparrow},
    \end{equation}
    where $0 < \delta < s$ and
    \begin{align*}
    \psi[V^{s / \delta}v^{1 - s / \delta};s/\delta](x) & \approx \bigg(\int_x^{\infty}V^{-(s / \delta)'}v\bigg)^{- \frac{(s / \delta)'}{(s / \delta)' + 1}}V^{-(s / \delta)'}(x)v(x), ~(x > 0), \\
    \Psi [V^{s / \delta}v^{1 - s / \delta};s/\delta](x) & \approx \bigg(\int_x^{\infty} V^{-(s / \delta)'}v\bigg)^{\frac{1}{(s / \delta)' + 1}},~(x > 0)
    \end{align*}
    and \eqref{Reduction.Theorem.eq.1.3} hold.
\end{theorem}
\begin{proof}
    Inequality \eqref{Reduction.Theorem.eq.1.1} is equivalent to
    \begin{equation}\label{RT.thm.main.9.eq.2}
    \left\|  \left\{T(f^{1/\delta})\right\}^{\delta} \right\|_{\beta / \delta, w, (0,\infty)} \le c^{\delta} \|f\|_{s / \delta,v,(0,\infty)}, \, f \in {\mathfrak M}^{\downarrow}.
    \end{equation}
    By Theorems \ref{Reduction.Theorem.thm.3.1},
    \eqref{RT.thm.main.9.eq.2} holds if and only if
    \begin{equation}\label{RT.thm.main.9.eq.3}
    \left\|  \left\{T\left(\int_x^{\infty} h \right)^{1/\delta}\right\}^{\delta} \right\|_{\beta / \delta, w, (0,\infty)} \le c^{\delta} \|h\|_{s / \delta,V^{s / \delta}v^{1 - s / \delta},(0,\infty)}, \, h \in {\mathfrak M}^{+},
    \end{equation}
    and \eqref{Reduction.Theorem.eq.1.3} hold.
    By Theorem \ref{RT.thm.main.4}, \eqref{RT.thm.main.9.eq.3} is equivalent to
    \begin{align}\label{RT.thm.main.9.eq.4}
    \bigg\|  \bigg\{T_{\big\{\Psi\big[V^{s / \delta}v^{1 - s / \delta};s/\delta\big]\big\}^2} \big( f^{1/\delta} \big)\bigg\}^{\delta} \bigg\|_{\beta / \delta, w, (0,\infty)}  \le c^{\delta} \|f\|_{s / \delta,\psi\big[V^{s / \delta}v^{1 - s / \delta};s/\delta\big],(0,\infty)}, \, f \in {\mathfrak M}^{\uparrow},
    \end{align}
    with
    \begin{align*}
    \psi\big[V^{s / \delta}v^{1 - s / \delta};s/\delta\big] & \approx (V^{1-(s / \delta)'} - V^{1-(s / \delta)'}(\infty))^{-(s / \delta)' / ((s / \delta)' + 1)}V^{-(s / \delta)'}v \\
    \Psi\big[V^{s / \delta}v^{1 - s / \delta};s/\delta\big] & \approx (V^{1-(s / \delta)'} - V^{1-(s / \delta)'}(\infty))^{1 / ((s / \delta)' + 1)}.
    \end{align*}
    Note that  \eqref{RT.thm.main.9.eq.4} is equivalent to \eqref{RT.thm.main.9.eq.1}, and this completes the proof.
\end{proof}

To state the next statements we need the following notations:
$$
V_1 (x) : = \bigg( \int_x^{\infty} V^{-2}v \bigg)^{1/3}, \qquad (x >
0).
$$

The following statement holds true.
\begin{corollary}\label{HardyIneqOnTheCone.cor.17}
    Let $0 < \beta \le \infty$, $0 < s < \infty$, and let $T: {\mathfrak M}^+ \rightarrow {\mathfrak M}^+$
    satisfies conditions {\rm (i)-(iii)}. Assume that $u,\,w \in {\mathcal W}(0,\infty)$
    and $v \in {\mathcal W}(0,\infty)$ be such that $V(x) < \infty$ for all $x > 0$ holds. Then
    inequality \eqref{Reduction.Theorem.eq.1.1} holds if and only if
    both
    \begin{equation}\label{HardyIneqOnTheCone.cor.17.eq/1}
    \left\|  T_{\{\Psi [V^2v^{-1};2]\}^{4 / s}} (f) \right\|_{\beta, w, (0,\infty)} \le c \|f\|_{s,\psi[V^{2}v^{-1};2],(0,\infty)}, \, f \in {\mathfrak M}^{\uparrow},
    \end{equation}
    where
    \begin{align*}
    \psi[V_*^{2}v^{-1};2](x) & \approx \{V_1 \cdot V\}^{- 2} (x)v(x), ~(x>0), \\
    \Psi [V_*^{2}v^{-1};2](x) & \approx V_1(x),~(x>0),
    \end{align*}
    and \eqref{Reduction.Theorem.eq.1.3} hold.
\end{corollary}
\begin{proof}
    The statement follows by Theorem \ref{RT.thm.main.9} with $\delta = s /2$.
\end{proof}

The following "dual" statement also holds true and can be proved
analogously.
\begin{theorem}\label{RT.thm.main.10}
    Let $0 < \beta \le \infty$, $0 < s < \infty$, and let $T: {\mathfrak M}^+ \rightarrow {\mathfrak M}^+$
    satisfies conditions {\rm (i)-(iii)}. Assume that $u,\,w \in {\mathcal W}(0,\infty)$
    and $v \in {\mathcal W}(0,\infty)$ be such that $V_*(x) < \infty$ for all $x > 0$ holds. Then
    inequality \eqref{Reduction.Theorem.eq.1.1.00} holds if and only if
    both
    \begin{equation}\label{RT.thm.main.10.eq.1}
    \left\|  T_{\big\{\Phi \big[V_*^{s / \delta}v^{1 - s / \delta};s/\delta\big]\big\}^{2 / \delta}} (f) \right\|_{\beta, w, (0,\infty)} \le c \|f\|_{s,\phi\big[V_*^{s / \delta}v^{1 - s / \delta};s/\delta\big],(0,\infty)}, \, f \in {\mathfrak M}^{\downarrow},
    \end{equation}
    where $0 < \delta < s$ and
    \begin{align*}
    \phi[V_*^{s / \delta}v^{1 - s / \delta};s/\delta](x) & \approx \bigg( \int_0^x V_*^{-(s / \delta)'}v\bigg)^{- \frac{(s / \delta)'}{(s / \delta)' + 1}}V_*^{-(s / \delta)'}(x)v(x),~(x>0), \\
    \Phi [V_*^{s / \delta}v^{1 - s / \delta};s/\delta](x) & \approx \bigg(\int_0^x V_*^{-(s / \delta)'}v\bigg)^{\frac{1}{(s / \delta)' + 1}},~(x>0),
    \end{align*}
    and \eqref{Reduction.Theorem.eq.1.3} hold.
\end{theorem}

To state the next statement we need the following notations:
$$
V_1^* (x) : = \bigg( \int_0^x V_*^{-2}v \bigg)^{1/3},\qquad (x > 0).
$$
\begin{corollary}\label{HardyIneqOnTheCone.cor.16}
    Let $0 < \beta \le \infty$, $0 < s < \infty$, and let $T: {\mathfrak M}^+ \rightarrow {\mathfrak M}^+$
    satisfies conditions {\rm (i)-(iii)}. Assume that $u,\,w \in {\mathcal W}(0,\infty)$
    and $v \in {\mathcal W}(0,\infty)$ be such that $V_*(x) < \infty$ for all $x > 0$ holds. Then
    inequality \eqref{Reduction.Theorem.eq.1.1.00} holds if and only if
    both
    \begin{equation}\label{HardyIneqOnTheCone.cor.16.eq/1}
    \left\|  T_{\big\{\Phi \big[V_*^2v^{-1};2\big]\big\}^{4 / p}} (f) \right\|_{\beta, w, (0,\infty)} \le c \|f\|_{s,\phi\big[V_*^{2}v^{-1};2\big],(0,\infty)}, \, f \in {\mathfrak M}^{\downarrow},
    \end{equation}
    where
    \begin{align*}
    \phi[V_*^{2}v^{-1};2](x) & \approx \{ V_1^* \cdot V_*\}^{-2}(x)v(x), ~(x>0), \\
    \Phi [V_*^{2}v^{-1};2](x) & \approx V_1^* (x),~(x>0),
    \end{align*}
    and \eqref{Reduction.Theorem.eq.1.3} hold.
\end{corollary}

% % % % % % % % % % % % % % % % % % % % % % % % % % %

% % % % % % % % % % % % % % % % % % % % % % % % % % %

\section{The weighted Hardy-type inequalities on the cones of monotone functions}\label{HardyIneqOnTheCone}

In this section we consider weighted Hardy inequalities on the cones
of monotone functions.

Note that inequality
\begin{equation}\label{Reduction.Theorem.Thm.2.5}
\| H_u (f) \|_{q,w,(0,\infty)} \le c \| f \|_{p,v,(0,\infty)}, \, f
\in {\mathfrak M}^{\downarrow}
\end{equation}
was considered by many authors and there exist several
characterizations of this inequality (see, survey paper \cite{cpss},
\cite{bengros}, \cite{gjop}, \cite{cgmp2008},  and \cite{GogStep}).

Using change of variables $x = 1/t$, we can easily obtain full
characterization of the weighted inequality
\begin{equation}\label{Reduction.Theorem.Thm.2.5.00}
\| H_u^* (f) \|_{q,w,(0,\infty)} \le c \| f \|_{p,v,(0.\infty)}, \,
f \in {\mathfrak M}^{\uparrow}.
\end{equation}

Our aim in this section is to give the characterization of the
inequalities
\begin{equation}\label{HardyIneqOnTheCone.cor.38.eq.1}
\| H_u (f) \|_{q,w,(0,\infty)} \le c \| f \|_{p,v,(0,\infty)}, \, f
\in {\mathfrak M}^{\uparrow}
\end{equation}
and
\begin{equation}\label{HardyIneqOnTheCone.cor.18.eq.1}
\| H_u^* (f) \|_{q,w,(0,\infty)} \le c \| f \|_{p,v,(0,\infty)}, \,
f \in {\mathfrak M}^{\downarrow}.
\end{equation}

Inequality \eqref{HardyIneqOnTheCone.cor.38.eq.1} was considered in
\cite{heinstep1993} in the case when $1 < p,\,q < \infty$, and
recently, completely characterized in \cites{gold2011.1,gold2011.2}
and \cite{GogStep} in the case $0 < p,\,q < \infty$. It is worth to
mention that in the most difficult case when $0 < q < p \le 1$, the
characterization obtained in \cite[Theorem 3.12]{GogStep} involves
additional function $\varphi (x) : = W^{-1}(4W(x))$, where
$W^{-1}(t) : = \inf\{s \ge 0:\, W(s) = t\}$ is the generalized
inverse function of $W$. Theorem \ref{Thm.2.5.0} give us a another
characterization of \eqref{HardyIneqOnTheCone.cor.38.eq.1} and its
proof does not use the discretization technique.

Recall the following complete characterization of the weighted Hardy
inequality on the cone of non-increasing functions.
\begin{theorem}[\cite{GogStep}, Theorems 2.5, 3.15, 3.16]\label{Thm.2.5.}
    Let $0 < q,\,p \le \infty$. Then inequality \eqref{Reduction.Theorem.Thm.2.5}
    with the best constant $c$ holds if and only if:

    {\rm (i)} $1 < p \le q < \infty$, and in this case $c \approx A_0 + A_1$,
    where
    \begin{align*}
    A_0 : & = \sup_{t > 0}\bigg( \int_0^t U^q(\tau) w(\tau)\,d\tau \bigg)^{\frac{1}{q}} V^{- \frac{1}{p}} (t),\\
    A_1 : & = \sup_{t > 0} W_*^{\frac{1}{q}}(t) \bigg( \int_0^t \bigg( \frac{U(\tau)}{V(\tau)}\bigg)^{p'}v(\tau)\,d \tau \bigg)^{\frac{1}{p'}};
    \end{align*}

    {\rm (ii)} $q < p < \infty$ and $1 < p < \infty$, and in this case $c \approx
    B_0 + B_1$, where
    \begin{align*}
    B_0 : & =  \bigg( \int_0^{\infty} V^{- \frac{r}{p}}(t)\bigg( \int_0^t U^q(\tau) w(\tau)\,d\tau
    \bigg)^{\frac{r}{p}} U^q (t) w(t)\,dt \bigg)^{\frac{1}{r}}, \\
    B_1 : & = \bigg( \int_0^{\infty} W_*^{\frac{r}{p}}(t) \bigg( \int_0^t \bigg(
    \frac{U(\tau)}{V(\tau)} \bigg)^{p'} v(\tau)\,d\tau \bigg)^{\frac{r}{p'}} w(t)\,dt \bigg)^{\frac{1}{r}};
    \end{align*}

    {\rm (iii)} $q < p \le 1$, and in this case $c \approx B_0 + C_1$, where
    \begin{align*}
    C_1 : & =  \bigg( \int_0^{\infty} \bigg( \esup_{\tau \in (0,t)}
    \frac{U^p(\tau)}{V(\tau)}\bigg)^{\frac{r}{p}} W_*^{\frac{r}{p}}(t) w(t) \,dt \bigg)^{\frac{1}{r}};
    \end{align*}

    {\rm (iv)} $p \le q < \infty$ and $p \le 1$, and in this case $c =
    D_0$, where
    $$
    D_0 : = \sup_{t > 0} V^{- \frac{1}{p}}(t)\bigg( \int_0^{\infty} U^q (\min
    \{\tau,t\}) w(\tau)\,d\tau \bigg)^{ \frac{1}{q}};
    $$

    {\rm (v)} $p \le 1$ and $q = \infty$, and in this case $c = E_0$, where
    $$
    E_0 : = \esup_{t > 0} V^{- \frac{1}{p}}(t) \bigg( \esup_{\tau > 0} \, U
    (\min \{\tau,t\}) w(\tau) \bigg);
    $$

    {\rm (vi)} $1 < p < \infty$ and $q = \infty$, and in this case $c = F_0$,
    where
    $$
    F_0 : = \esup_{t > 0} w(t) \bigg( \int_0^t \bigg( \int_{\tau}^t u(y)
    V^{- 1}(y)\,dy \bigg)^{p'}v(\tau)\,d\tau \bigg)^{\frac{1}{p'}};
    $$

    {\rm (vii)} $p = \infty$ and $0 < q < \infty$, and in this case $c = G_0$,
    where
    $$
    G_0 : = \bigg( \int_0^{\infty} \bigg( \int_0^t
    \frac{u(y)\,dy}{\esup_{\tau \in (0,y)} v(\tau)}\bigg)^q
    w(t)\,dt\bigg)^{\frac{1}{q}};
    $$

    {\rm (viii)} $p = q = \infty$, and in this case $c = H_0$,
    where
    $$
    H_0 : = \esup_{t > 0} \bigg( \int_0^t \frac{u(y)\,dy}{\esup_{\tau
            \in (0,y)} v(\tau)}\bigg) w(t).
    $$
\end{theorem}

The following theorem holds true.
\begin{theorem}\label{Thm.2.5.00}
    Let $0 < q,\,p \le \infty$. Then inequality \eqref{Reduction.Theorem.Thm.2.5.00}
    with the best constant $c$ holds if and only if:

    {\rm (i)} $1 < p \le q < \infty$, and in this case $c \approx A_0^* +
    A_1^*$, where
    \begin{align*}
    A_0^* : & = \sup_{t > 0}\bigg( \int_t^{\infty} U_*^q(\tau) w(\tau)\,d\tau \bigg)^{\frac{1}{q}} V_*^{- \frac{1}{p}} (t),\\
    A_1^* : & = \sup_{t > 0} W^{\frac{1}{q}}(t) \bigg( \int_t^{\infty} \bigg(
    \frac{U_*(\tau)}{V_*(\tau)}\bigg)^{p'}v(\tau)\,d\tau \bigg)^{\frac{1}{p'}};
    \end{align*}

    {\rm (ii)} $q < p < \infty$ and $1 < p < \infty$, and in this case $c
    \approx B_0^* + B_1^*$, where
    \begin{align*}
    B_0^* : & =  \bigg( \int_0^{\infty} V_*^{- \frac{r}{p}}(t)\bigg( \int_t^{\infty}
    U_*^q(\tau) w(\tau)\,d\tau \bigg)^{\frac{r}{p}} U_*^q (t) w(t)\,dt \bigg)^{\frac{1}{r}}, \\
    B_1^* : & = \bigg( \int_0^{\infty} W^{\frac{r}{p}}(t) \bigg( \int_t^{\infty}
    \bigg( \frac{U_*(\tau)}{V_*(\tau)} \bigg)^{p'} v(\tau)\,d\tau \bigg)^{\frac{r}{p'}} w(t)\,dt
    \bigg)^{\frac{1}{r}};
    \end{align*}

    {\rm (iii)} $q < p \le 1$, and in this case $c \approx B_0^* +  C_1^*$, where
    \begin{align*}
    C_1^* : & =  \bigg( \int_0^{\infty} \bigg( \esup_{y \in (t,\infty)}
    \frac{U_*^p(y)}{V_*(y)}\bigg)^{\frac{r}{p}} W^{\frac{r}{p}}(t) w(t) \,dt
    \bigg)^{\frac{1}{r}};
    \end{align*}

    {\rm (iv)} $p \le q < \infty$ and $p \le 1$, and in this case $c
    = D_0^*$, where
    $$
    D_0^* : = \sup_{t > 0} V_*^{- \frac{1}{p}}(t)\bigg( \int_0^{\infty} U_*^q
    (\max \{\tau,t\}) w(\tau)\,d\tau \bigg)^{ \frac{1}{q}}.
    $$

    {\rm (v)} $p \le 1$ and $q = \infty$, and in this case $c = E_0$, where
    $$
    E_0^* : = \esup_{t > 0} V_*^{- \frac{1}{p}}(t) \bigg( \esup_{\tau > 0} \, U_*
    (\max \{\tau,t\}) w(\tau) \bigg);
    $$

    {\rm (vi)} $1 < p < \infty$ and $q = \infty$, and in this case $c = F_0^*$,
    where
    $$
    F_0^* : = \esup_{t > 0} w(t) \bigg( \int_t^{\infty} \bigg( \int_t^{\tau} u(y)
    V_*^{- 1}(y)\,dy \bigg)^{p'}v(\tau)\,d\tau \bigg)^{\frac{1}{p'}};
    $$

    {\rm (vii)} $p = \infty$ and $0 < q < \infty$, and in this case $c = G_0^*$,
    where
    $$
    G_0^* : = \bigg( \int_0^{\infty} \bigg( \int_t^{\infty}
    \frac{u(y)\,dy}{\esup_{\tau \in (y,\infty)} v(\tau)}\bigg)^q
    w(t)\,dt\bigg)^{\frac{1}{q}};
    $$

    {\rm (viii)} $p = q = \infty$, and in this case $c = H_0^*$,
    where
    $$
    H_0^* : = \esup_{t > 0} \bigg( \int_t^{\infty} \frac{u(y)\,dy}{\esup_{\tau \in (y,\infty)} v(\tau)}\bigg) w(t).
    $$
\end{theorem}
\begin{proof}
    By change of variables $x = 1 / t$, it is easy to see that inequality \eqref{Reduction.Theorem.Thm.2.5.00} holds if and only if
    \begin{equation*}
    \left\|H_{p,\tilde{u}} (f)\right\|_{q,\tilde{w},(0,\infty)} \leq c
    \,\|f\|_{p,\tilde{v},(0,\infty)}, \quad f \in {\mathfrak M}^{\downarrow}
    \end{equation*}
    holds, where
    $$
    \tilde{u} (t) = u \bigg(\frac{1}{t}\bigg)\frac{1}{t^2}, ~ \tilde{w} (t) = w \bigg(\frac{1}{t}\bigg)\frac{1}{t^2}, ~\tilde{v} (t) = v \bigg(\frac{1}{t}\bigg)\bigg(\frac{1}{t^2}\bigg), ~ t > 0,
    $$
    when $0 < p < \infty$, $0 < q < \infty$, and
    $$
    \tilde{u} (t) = u \bigg(\frac{1}{t}\bigg)\frac{1}{t^2}, ~ \tilde{w} (t) = w \bigg(\frac{1}{t}\bigg), ~\tilde{v} (t) = v \bigg(\frac{1}{t}\bigg)\bigg(\frac{1}{t^2}\bigg), ~ t > 0,
    $$
    when $0 < p < \infty$, $q = \infty$, and
    $$
    \tilde{u} (t) = u \bigg(\frac{1}{t}\bigg)\frac{1}{t^2}, ~ \tilde{w} (t) = w \bigg(\frac{1}{t}\bigg)\bigg(\frac{1}{t^2}\bigg), ~\tilde{v} (t) = v \bigg(\frac{1}{t}\bigg), ~ t > 0,
    $$
    when $p = q = \infty$,
    and
    $$
    \tilde{u} (t) = u \bigg(\frac{1}{t}\bigg)\frac{1}{t^2}, ~ \tilde{w} (t) = w \bigg(\frac{1}{t}\bigg), ~\tilde{v} (t) = v \bigg(\frac{1}{t}\bigg), ~ t > 0.
    $$

    Using Theorem \ref{Thm.2.5.}, and then applying substitution of variables mentioned above three times, we get the statement.
\end{proof}

The following theorem is true.
\begin{theorem}\label{Thm.2.5.0}
    Let $0 < q \le \infty$ and $0 < p < \infty$. Assume that $u,\,w \in {\mathcal W}(0,\infty)$
    and $v \in {\mathcal W}(0,\infty)$ be such that $V_*(x) < \infty$ for all $x > 0$ holds.
    Recall that
    $$
    V_1^* (x) : = \bigg( \int_0^x V_*^{-2}v \bigg)^{1/3},\qquad (x > 0).
    $$
    Denote by
    \begin{align*}
    U_1^*(x) : = \int_0^x u(t) [V_1^*]^{ \frac{4}{p}}(t) \,dt, \qquad (x > 0).
    \end{align*}
    Then inequality \eqref{HardyIneqOnTheCone.cor.38.eq.1}
    with the best constant $c$ holds if and only if:

    {\rm (i)} $1 < p \le q < \infty$, and in this case
    $$
    c \approx \tilde{A}_0 + \tilde{A}_1 + \| H_u ({\bf 1}) \|_{q,w,(0,\infty)} / \| {\bf 1} \|_{p,v,(0,\infty)},
    $$
    where
    \begin{align*}
    \tilde{A}_0 : & = \sup_{t > 0}\bigg( \int_0^t [U_1^*]^q (\tau) w(\tau)\,d\tau \bigg)^{\frac{1}{q}}  [V_1^*]^{- \frac{1}{p}}(t),\\
    \tilde{A}_1 : & = \sup_{t > 0} W_*^{\frac{1}{q}}(t) \bigg( \int_0^t [U_1^*]^{p'} (\tau) [V_1^*]^{- (2 + p')}(\tau) V_*^{-2}(\tau)v(\tau)\,d\tau \bigg)^{\frac{1}{p'}};
    \end{align*}

    {\rm (ii)} $q < p < \infty$ and $1 < p < \infty$, and in this case
    $$
    c \approx
    \tilde{B}_0 + \tilde{B}_1 + \| H_u ({\bf 1}) \|_{q,w,(0,\infty)} / \| {\bf 1} \|_{p,v,(0,\infty)},
    $$
    where
    \begin{align*}
    \tilde{B}_0 : & =  \bigg( \int_0^{\infty} [V_1^*]^{- \frac{r}{p}}(t)\bigg( \int_0^t [U_1^*]^q (\tau) w(\tau)\,d\tau \bigg)^{\frac{r}{p}} [U_1^*]^q (t) w(t)\,dt \bigg)^{\frac{1}{r}}, \\
    \tilde{B}_1 : & = \bigg( \int_0^{\infty} W_*^{\frac{r}{p}}(t) \bigg( \int_0^t  [U_1^*]^{p'}(\tau) [V_1^*]^{- (2 + p')}(\tau) V_*^{-2}(\tau) v(\tau)\,d\tau
    \bigg)^{\frac{r}{p'}} w(t)\,dt \bigg)^{\frac{1}{r}};
    \end{align*}

    {\rm (iii)} $q < p \le 1$, and in this case
    $$
    c \approx \tilde{B}_0 + \tilde{C}_1 + \| H_u ({\bf 1}) \|_{q,w,(0,\infty)} / \| {\bf 1} \|_{p,v,(0,\infty)},
    $$
    where
    \begin{align*}
    \tilde{C}_1 : & =  \bigg( \int_0^{\infty} \bigg( \esup_{\tau \in [0,t]} \frac{[U_1^*]^p(\tau)}{V_1^*(\tau)}
    \bigg)^{\frac{r}{p}} W_*^{\frac{r}{p}}(t) w(t) \,dt \bigg)^{\frac{1}{r}};
    \end{align*}

    {\rm (iv)} $p \le q < \infty$ and $0 < p \le 1$, and in this case
    $$
    c = \tilde{D}_0 + \| H_u ({\bf 1}) \|_{q,w,(0,\infty)} / \| {\bf 1} \|_{p,v,(0,\infty)},
    $$
    where
    $$
    \tilde{D}_0 : = \sup_{t > 0} [V_1^*]^{- \frac{1}{p}}(t)\bigg( \int_0^{\infty} [U_1^*]^q (\min
    \{\tau,t\}) w(\tau)\,d\tau \bigg)^{ \frac{1}{q}};
    $$

    {\rm (v)} $p \le 1$ and $q = \infty$, and in this case $$
    c = \tilde{E}_0 + \| H_u ({\bf 1}) \|_{q,w,(0,\infty)} / \| {\bf 1} \|_{p,v,(0,\infty)},
    $$
    where
    $$
    \tilde{E}_0 : = \esup_{t > 0} [V_1^*]^{- \frac{1}{p}}(t) \bigg( \esup_{\tau > 0} \, [U_1^*]
    (\min \{\tau,t\}) w(\tau) \bigg);
    $$

    {\rm (vi)} $1 < p < \infty$ and $q = \infty$, and in this case
    $$
    c = \widetilde{F}_0 + \| H_u ({\bf 1}) \|_{q,w,(0,\infty)} / \| {\bf 1} \|_{p,v,(0,\infty)},
    $$
    where
    $$
    \tilde{F}_0 : = \esup_{t > 0} w(t) \bigg( \int_0^t \bigg( \int_{\tau}^t u(y)
    [V_1^*]^{\frac{4 - p}{p}}(y)\,dy \bigg)^{p'}[V_1^*]^{- 2}(\tau)V_*^{-2}(\tau)v(\tau)\,d\tau \bigg)^{\frac{1}{p'}}.
    $$
\end{theorem}
\begin{proof}
    By Corollary \ref{HardyIneqOnTheCone.cor.16} applied with $\beta = q$, $s = p$ and $T = H_u$, inequality \eqref{HardyIneqOnTheCone.cor.38.eq.1} holds if and only if
    both
    \begin{equation}\label{HardyIneqOnTheCone.cor.18.eq.2}
    \left\|  H_{ u [V_1^*]^{4 / p}} (f) \right\|_{q, w, (0,\infty)} \le c \|f\|_{p,\{ V_1^* \cdot V_*\}^{-2}v,(0,\infty)}, \, f \in {\mathfrak M}^{\downarrow},
    \end{equation}
    and
    \begin{equation}\label{HardyIneqOnTheCone.cor.18.eq.3}
    \| H_u ({\bf 1}) \|_{q,w,(0,\infty)} \le c \| {\bf 1} \|_{p,v,(0,\infty)}
    \end{equation}
    hold.

    Now the statement follows by applying Theorem \ref{Thm.2.5.}.
\end{proof}

\begin{theorem}\label{Thm.2.5.0000}
    Let $0 < q \le \infty$ and $0 < p < \infty$.
    Recall that
    $$
    V_1 (x) : = \bigg( \int_x^{\infty} V^{-2}v \bigg)^{\frac{1}{3}}, \qquad (x > 0).
    $$
    Denote by
    \begin{align*}
    U_1(x) : = \int_x^{\infty} u(t) V_1^{\frac{4}{p}}(t) \,dt, \qquad (x > 0).
    \end{align*}
    Then inequality
    \eqref{HardyIneqOnTheCone.cor.18.eq.1}
    with the best constant $c$ holds if and only if:

    {\rm (i)} $1 < p \le q < \infty$, and in this case
    $$
    c \approx \tilde{A}_0^* +
    \tilde{A}_1^* + \| H_u^* ({\bf 1}) \|_{q,w,(0,\infty)} / \| {\bf 1} \|_{p,v,(0,\infty)},
    $$
    where
    \begin{align*}
    \tilde{A}_0^* : & = \sup_{t > 0}\bigg( \int_t^{\infty} U_1^q(\tau) w(\tau)\,d\tau \bigg)^{\frac{1}{q}} V_1^{- \frac{1}{p}} (t),\\
    \tilde{A}_1^* : & = \sup_{t > 0} W^{\frac{1}{q}}(t) \bigg( \int_t^{\infty} U_1^{p'}(\tau) V_1^{-(2+p')}(\tau) V^{-2}(\tau)v(\tau)\,d\tau \bigg)^{\frac{1}{p'}};
    \end{align*}

    {\rm (ii)} $q < p < \infty$ and $1 < p < \infty$, and in this case
    $$
    c
    \approx \tilde{B}_0^* + \tilde{B}_1^* + \| H_u^* ({\bf 1}) \|_{q,w,(0,\infty)} / \| {\bf 1} \|_{p,v,(0,\infty)},
    $$
    where
    \begin{align*}
    \tilde{B}_0^* : & =  \bigg( \int_0^{\infty} V_1^{- \frac{r}{p}}(t)\bigg( \int_t^{\infty}
    U_1^q(\tau) w(\tau)\,d\tau \bigg)^{\frac{r}{p}} U_1^q (t) w(t)\,dt \bigg)^{\frac{1}{r}}, \\
    \tilde{B}_1^* : & = \bigg( \int_0^{\infty} W^{\frac{r}{p}}(t) \bigg( \int_t^{\infty}
    U_1^{p'}(\tau) V_1^{-(2+p')}(\tau) V^{-2}(\tau)v(\tau)\,d\tau \bigg)^{\frac{r}{p'}} w(t)\,dt \bigg)^{\frac{1}{r}};
    \end{align*}

    {\rm (iii)} $q < p \le 1$, and in this case
    $$
    c \approx \tilde{B}_0^* +   \tilde{C}_1^* + \| H_u^* ({\bf 1}) \|_{q,w,(0,\infty)} / \| {\bf 1} \|_{p,v,(0,\infty)},
    $$
    where
    \begin{align*}
    \tilde{C}_1^* : & =  \bigg( \int_0^{\infty} \bigg( \esup_{\tau \in (t,\infty)}
    \frac{U_1^p(\tau)}{V_1(\tau)}\bigg)^{\frac{r}{p}} W^{\frac{r}{p}}(t) w(t) \,dt
    \bigg)^{\frac{1}{r}};
    \end{align*}

    {\rm (iv)} $p \le q < \infty$ and $p \le 1$, and in this case
    $$
    c = \tilde{D}_0^* + \| H_u^* ({\bf 1}) \|_{q,w,(0,\infty)} / \| {\bf 1} \|_{p,v,(0,\infty)},
    $$
    where
    $$
    \tilde{D}_0^* : = \sup_{t > 0} V_1^{- \frac{1}{p}}(t)\bigg( \int_0^{\infty} U_1^q
    (\max \{s,t\}) w(s)\,ds \bigg)^{ \frac{1}{q}}.
    $$

    {\rm (v)} $p \le 1$ and $q = \infty$, and in this case $$
    c = \tilde{E}_0 + \| H_u^* ({\bf 1}) \|_{q,w,(0,\infty)} / \| {\bf 1} \|_{p,v,(0,\infty)},
    $$
    where
    $$
    \tilde{E}_0^* : = \esup_{t > 0} V_1^{- \frac{1}{p}}(t) \bigg( \esup_{\tau > 0} \, U_1
    (\max \{\tau,t\}) w(\tau) \bigg);
    $$

    {\rm (vi)} $1 < p < \infty$ and $q = \infty$, and in this case
    $$
    c = \tilde{F}_0^* + \| H_u^* ({\bf 1}) \|_{q,w,(0,\infty)} / \| {\bf 1} \|_{p,v,(0,\infty)},
    $$
    where
    $$
    \tilde{F}_0^* : = \esup_{t > 0} w(t) \bigg( \int_t^{\infty} \bigg( \int_t^{\tau} u(y)
    V_1^{- 1}(y)\,dy \bigg)^{p'} V_1^{-2}(\tau) V^{-2}(\tau)v(\tau)\,d\tau \bigg)^{\frac{1}{p'}}.
    $$
\end{theorem}
\begin{proof}
    By change of variables $x = 1 / t$, it is easy to see that inequality     \eqref{HardyIneqOnTheCone.cor.18.eq.1} holds if and only if
    \begin{equation*}
    \left\|H_{p,\tilde{u}} (f)\right\|_{q,\tilde{w},(0,\infty)} \leq c
    \,\|f\|_{p,\tilde{v},(0,\infty)}, \quad f \in {\mathfrak M}^{\uparrow}
    \end{equation*}
    holds, where
    $$
    \tilde{u} (t) = u \bigg(\frac{1}{t}\bigg)\frac{1}{t^2}, ~ \tilde{w} (t) = w \bigg(\frac{1}{t}\bigg)\frac{1}{t^2}, ~\tilde{v} (t) = v \bigg(\frac{1}{t}\bigg)\bigg(\frac{1}{t^2}\bigg), ~ t > 0,
    $$
    when $0 < p < \infty$, $0 < q < \infty$, and
    $$
    \tilde{u} (t) = u \bigg(\frac{1}{t}\bigg)\frac{1}{t^2}, ~ \tilde{w} (t) = w \bigg(\frac{1}{t}\bigg), ~\tilde{v} (t) = v \bigg(\frac{1}{t}\bigg)\bigg(\frac{1}{t^2}\bigg), ~ t > 0,
    $$
    when $0 < p < \infty$, $q = \infty$, and
    $$
    \tilde{u} (t) = u \bigg(\frac{1}{t}\bigg)\frac{1}{t^2}, ~ \tilde{w} (t) = w \bigg(\frac{1}{t}\bigg)\bigg(\frac{1}{t^2}\bigg), ~\tilde{v} (t) = v \bigg(\frac{1}{t}\bigg), ~ t > 0,
    $$
    when $p = q = \infty$,
    and
    $$
    \tilde{u} (t) = u \bigg(\frac{1}{t}\bigg)\frac{1}{t^2}, ~ \tilde{w} (t) = w \bigg(\frac{1}{t}\bigg), ~\tilde{v} (t) = v \bigg(\frac{1}{t}\bigg), ~ t > 0.
    $$

    Using Theorem \ref{Thm.2.5.0}, and then applying substitution of variables mentioned above three times, we get the statement.
\end{proof}

% % % % % % % % % % % % % % % % % % % % % % % % % % %

% % % % % % % % % % % % % % % % % % % % % % % % % % %

\section{The weighted norm inequalities for  iterated Hardy-type operators}\label{RT.SC}

In this section we give complete characterization of inequalities
\eqref{IHI.1} - \eqref{IHI.2} and \eqref{IHI.3} - \eqref{IHI.4}.

Using results obtained in the previous section we can reduce the
characterization of inequality \eqref{IHI.1} to the weighted Hardy
inequality on the cones of non-increasing functions.

The following theorem is true.
\begin{theorem}\label{RT.SC.thm.1.1}
    Let $0 < p < \infty$, $0 < q \le \infty$ and $1 < s < \infty$. Assume that $u,\,w \in {\mathcal W}(0,\infty)$
    and $v \in {\mathcal W}(0,\infty)$ be such that \eqref{RT.thm.main.3.eq.0} holds.
    Recall that
    $$
    \Phi \big[v;s\big](x) = \bigg( \int_0^x
    v^{1-{s}^{\prime}}(t)\,dt\bigg)^{\frac{1}{s^{\prime}+ 1}}, ~ x > 0.
    $$
    Denote by
    $$
    \Phi_1(\tau) : =  \int_0^{\tau} u(x) \Phi \big[v;s\big]^{2p} (x)\,dx = \int_0^{\tau} u(x) \bigg( \int_0^x
    v^{1-{s}^{\prime}}(t)\,dt\bigg)^{\frac{2p}{s^{\prime}+ 1}}dx, ~ \tau > 0.
    $$
    Then inequality \eqref{IHI.1} with the best constant $c_1$ holds if
    and only if:

    {\rm (i)} $p < s \le q < \infty$, and in this case $c_1 \approx A_{1,1} +
    A_{1,2}$, where
    \begin{align*}
    A_{1,1} : & = \sup_{t > 0}\bigg( \int_0^t [\Phi_1]^{ \frac{q}{p}}(\tau)
    w(\tau)\,d\tau \bigg)^{\frac{1}{q}} \Phi[v;s]^{- \frac{1}{s}} (t),\\
    A_{1,2} : & = \sup_{t > 0} W_*^{\frac{1}{q}} (t) \bigg( \int_0^t \bigg(
    \frac{\Phi_1(\tau)}{\Phi[v;s](\tau)} \bigg)^{\frac{s}{s-p}} \phi[v;s](\tau)\,d\tau
    \bigg)^{\frac{s - p}{p s}};
    \end{align*}

    {\rm (ii)} $q  < s < \infty$ and $p < s$, and in this case $c_1 \approx
    B_{1,1} + B_{1,2}$, where
    \begin{align*}
    B_{1,1} : & =  \bigg( \int_0^{\infty} \Phi[v;s]^{\frac{q}{q - s}}(t)
    \bigg(\int_0^t [\Phi_1]^{\frac{q}{p}}(\tau) w(\tau)\,d\tau \bigg)^{\frac{q}{s - q}} [\Phi_1]^{\frac{q}{p}}(t) w(t)\,dt \bigg)^{\frac{s - q}{qs}},\\
    B_{1,2} : & = \bigg( \int_0^{\infty} W_*^{\frac{q}{s - q}}(t) \bigg(
    \int_0^t \bigg( \frac{\Phi_1(\tau)}{\Phi[v;s](\tau)} \bigg)^{\frac{s}{s - p}} \phi[v;s](\tau)\,d\tau \bigg)^{\frac{q (s - p)}{p (s - q)}} w(t)\,dt \bigg)^{\frac{s - q}{qs}};
    \end{align*}

    {\rm (iii)} $q < s \le p$, and in this case $c_1 \approx B_{1,1} + C_1$,
    where
    \begin{align*}
    C_1 : & =  \bigg( \int_0^{\infty} \bigg( \esup_{\tau \in (0,t)}
    \frac{[\Phi_1]^{\frac{s}{p}}(\tau)}{\Phi[v;s] (\tau)}\bigg)^{\frac{q}{s - q}} W_*^{\frac{q}{s - q}}(t) w(t)\,dt \bigg)^{\frac{s - q}{s q}};
    \end{align*}

    {\rm (iv)} $s \le q < \infty$ and $s \le p$, and in this case $c_1 =
    D_1$, where
    $$
    D_1 : = \sup_{t > 0}\Phi[v;s]^{- \frac{1}{s}}(t) \bigg( \int_0^{\infty} [\Phi_1]^{\frac{q}{p}}(\min\{\tau,t\}) w(\tau)\,d\tau \bigg)^{\frac{1}{q}};
    $$

    {\rm (v)} $s \le p$ and $q = \infty$, and in this case $c_1 = E_1$,
    where
    $$
    E_1 : = \esup_{t > 0}\Phi[v;s]^{- \frac{1}{s}}(t) \bigg( \esup_{\tau > 0} \,    \Phi_1(\min\{\tau,t\}) w(\tau) \bigg)^{\frac{1}{p}};
    $$

    {\rm (vi)} $p < s$ and $q = \infty$, and in this case $c_1 = F_1$, where
    $$
    F_1 : = \esup_{t > 0} w(t) \bigg( \int_0^t \bigg( \int_{\tau}^t u(y)
    \Phi[v;s]^{-1}(y)\,dy  \bigg)^{\frac{s}{s - p}}\phi[v;s](\tau)\,d\tau \bigg)^{\frac{s - p}{s p}}.
    $$
\end{theorem}

\begin{proof}
    By Theorem \ref{RT.thm.main.3} (with the operator $T = H_{p,u}$),
    ineqality \eqref{IHI.1} holds if and only if
    \begin{equation}\label{RT.SC.thm.1.eq.2}
    \bigg\| \int_0^x f u \Phi[v;s]^{2p} \bigg\|_{{q/p}, w, (0,\infty)} \le C_1^p \,\|
    f \|_{{s / p}, \phi[v;s],(0,\infty)}, \, f \in {\mathfrak M}^{\downarrow}
    \end{equation}
    holds. Moreover, $c_1 \approx C_1$. It remains to apply Theorem
    \ref{Thm.2.5.}.
\end{proof}

We have  the following statement when $s = 1$.
\begin{theorem}\label{RT.SC.thm.1.3}
    Let $0 < p < \infty$ and $0 < q \le \infty$. Assume that $u,\,w \in {\mathcal W}(0,\infty)$
    and $v \in {\mathcal W}(0,\infty)$ be such that $V(x) < \infty$ for all $x > 0$. Denote by
    $$
    V_2 (\tau) : = \int_0^{\tau} u(x) V^{2p}(x)\,dx, ~ \tau > 0.
    $$
    Then inequality
    \begin{equation}\label{RT.SC.thm.1.2.eq.1}
    \left\|H_{p,u} \bigg( \int_0^x h\bigg)\right\|_{q,w,(0,\infty)} \leq c_1^1
    \,\|h\|_{1,V^{-1},(0,\infty)},~h\in {\mathfrak M}^+
    \end{equation}
    with the best constant $c_1^1$ holds if
    and only if:

    {\rm (i)} $p < 1 \le q < \infty$, and in this case $c_1^1 \approx A_{1,1}^1 +
    A_{1,2}^1$, where
    \begin{align*}
    A_{1,1}^1 : & = \sup_{t > 0}\bigg( \int_0^t [V_2]^{\frac{q}{p}}(\tau)
    w(\tau)\,d\tau \bigg)^{\frac{1}{q}} V^{- 1} (t),\\
    A_{1,2}^1 : & = \sup_{t > 0} W_*^{\frac{1}{q}} (t) \bigg( \int_0^t \bigg(
    \frac{V_2(\tau)}{V(\tau)} \bigg)^{\frac{1}{1-p}} v(\tau)\,d\tau
    \bigg)^{\frac{1 - p}{p}};
    \end{align*}

    {\rm (ii)} $q  < 1$ and $p < 1$, and in this case $c_1^1 \approx
    B_{1,1}^1 + B_{1,2}^1$, where
    \begin{align*}
    B_{1,1}^1 : & =  \bigg( \int_0^{\infty} V^{\frac{q}{q - 1}}(t)
    \bigg(\int_0^t [V_2]^{\frac{q}{p}}(\tau) w(\tau)\,d\tau \bigg)^{\frac{q}{1 - q}} [V_2]^{ \frac{q}{p}}(t) w(t)\,dt \bigg)^{\frac{1 - q}{q}},\\
    B_{1,2}^1 : & = \bigg( \int_0^{\infty} W_*^{\frac{q}{1 - q}}(t) \bigg(
    \int_0^t \bigg( \frac{V_2(\tau)}{V(\tau)} \bigg)^{\frac{1}{1 - p}} v(\tau)\,d\tau \bigg)^{\frac{q (1 - p)}{p (1 - q)}} w(t)\,dt \bigg)^{\frac{1 - q}{q}};
    \end{align*}

    {\rm (iii)} $q < 1 \le p$, and in this case $c_1^1 \approx B_{1,1}^1 + C_1^1$,
    where
    \begin{align*}
    C_1^1 : & =  \bigg( \int_0^{\infty} \bigg( \esup_{\tau \in (0,t)}
    \frac{[V_2]^{\frac{1}{p}}(\tau)}{V (\tau)}\bigg)^{\frac{q}{1 - q}} W_*^{\frac{q}{1 - q}}(t) w(t)\,dt \bigg)^{\frac{1 - q}{q}};
    \end{align*}

    {\rm (iv)} $1 \le q < \infty$ and $1 \le p$, and in this case $c_1^1 = D_1^1$, where
    $$
    D_1^1 : = \sup_{t > 0} V^{- 1}(t) \bigg( \int_0^{\infty} [V_2]^{\frac{q}{p}}(\min\{\tau,t\}) w(\tau)\,d\tau \bigg)^{\frac{1}{q}};
    $$

    {\rm (v)} $1 \le p$ and $q = \infty$, and in this case $c_1^1 = E_1^1$,
    where
    $$
    E_1^1 : = \esup_{t > 0}V^{- 1}(t) \bigg( \esup_{\tau > 0} \,  V_2(\min\{\tau,t\}) w(\tau) \bigg)^{\frac{1}{p}};
    $$

    {\rm (vi)} $p < 1$ and $q = \infty$, and in this case $c_1^1 = F_1^1$, where
    $$
    F_1^1 : = \esup_{t > 0} w(t)^{\frac{1}{p}} \bigg( \int_0^t \bigg( \int_{\tau}^t u(y)
    V^{2p - 1}(y\,dy)  \bigg)^{\frac{1}{1 - p}}v(\tau)\,d\tau \bigg)^{\frac{1 - p}{p}}.
    $$
\end{theorem}

\begin{proof}
    By Theorem \ref{RT.thm.main.8.0} applied to the operator $H_{p,u}$,
    inequality \eqref{RT.SC.thm.1.2.eq.1} with the best constant $c_1$
    holds if and only if inequality
    \begin{equation}\label{RT.SC.thm.1.1.eq.2}
    \bigg\| \int_0^x f V^{2p}u \bigg\|_{{q/p}, w, (0,\infty)} \le C_1^p \,\|
    f \|_{{1 / p}, v,(0,\infty)}, \, f \in {\mathfrak M}^{\downarrow}
    \end{equation}
    holds. Moreover, $c_1 \approx C_1$. In order to complete the proof,
    it remains to apply Theorem \ref{Thm.2.5.}.
\end{proof}

The following theorems give us another more simpler and natural
method for characterization of inequality \eqref{IHI.2}, which is
different from that one worked out in \cite{GogMusPers1} and
\cite{GogMusPers2}.
\begin{theorem}\label{RT.SC.thm.2.2}
    Let $0 < p < \infty$, $0 < q \le \infty$ and $1 < s < \infty$. Assume that $u,\,w \in {\mathcal W}(0,\infty)$
    and $v \in {\mathcal W}(0,\infty)$ be such that \eqref{RT.thm.main.4.eq.0} holds.
    Denote by
    $$
    \Phi_2(\tau) : = \int_0^{\tau} u(x) \bigg(\Psi\big[v;s\big] \cdot \Phi \big[\Psi[v;s]^{s} \psi[v;s]^{1-s};s\big]\bigg)^{2p}(x)\,dx, ~ \tau > 0.
    $$
    Recall that
    \begin{align*}
    \Psi \big[v;s\big](x) & = \bigg( \int_x^{\infty}
    v^{1-{s}^{\prime}}(t)\,dt\bigg)^{\frac{1}{s^{\prime} + 1}},~ x > 0, \\
    \phi \big[\Psi[v;s]^{s} \psi[v;s]^{1-s};s\big] (x) & \\
    & \hspace{-3cm} \approx \bigg\{\int_0^x \bigg( \int_t^{\infty} v^{1 - s'}\bigg)^{- \frac{2s'}{1 + s'}} v^{1-s'}(t)\,dt\bigg\}^{- \frac{s'}{1 + s'}} \bigg( \int_x^{\infty} v^{1 - s'}\bigg)^{- \frac{2s'}{1 + s'}} v^{1-s'}(x), ~ x>0,\\
    \Phi \big[\Psi[v;s]^{s} \psi[v;s]^{1-s};s\big]  (x)
    & \approx \bigg\{\int_0^x \bigg( \int_t^{\infty} v^{1 - s'}\bigg)^{- \frac{2s'}{1 + s'}} v^{1-s'}(t)\,dt\bigg\}^{\frac{1}{1 + s'}}, ~ x>0.
    \end{align*}
    Then inequality \eqref{IHI.2} with the best constant $c_2$ holds if
    and only if:

    {\rm (i)} $p < s \le q < \infty$, and in this case
    $$
    c_2 \approx A_{2,1} +
    A_{2,2} + \|\| {\bf 1} \|_{p, \Psi[v;s]^{2p}u,(0,t)} \|_{q,w,(0,\infty)} /
    \| {\bf 1} \|_{s, \psi[v;s],(0,\infty)},
    $$
    where
    \begin{align*}
    A_{2,1} : & = \sup_{t > 0}\bigg( \int_0^t [\Phi_2]^{ \frac{q}{p}}(\tau)
    w(\tau)\,d\tau \bigg)^{\frac{1}{q}} \Phi \big[\Psi^s \psi^{1-s};s\big]^{- \frac{1}{s}} (t),\\
    A_{2,2} : & = \sup_{t > 0} W_*^{\frac{1}{q}} (t) \bigg( \int_0^t \bigg(
    \frac{\Phi_2(\tau)}{\Phi \big[\Psi^s \psi^{1-s};s\big](\tau)} \bigg)^{\frac{s}{s-p}} \phi \big[\Psi^s \psi^{1-s};s\big](\tau)\,d\tau
    \bigg)^{\frac{s - p}{p s}};
    \end{align*}

    {\rm (ii)} $q  < s < \infty$ and $p < s$, and in this case
    $$
    c_2 \approx
    B_{2,1} + B_{2,2} + \|\| {\bf 1} \|_{p, \Psi[v;s]^{2p}u,(0,t)} \|_{q,w,(0,\infty)} /
    \| {\bf 1} \|_{s, \psi[v;s],(0,\infty)},
    $$
    where
    \begin{align*}
    B_{2,1} : & =  \bigg( \int_0^{\infty} \Phi \big[\Psi^s \psi^{1-s};s\big]^{\frac{q}{q - s}}(t)
    \bigg(\int_0^t [\Phi_2]^{ \frac{q}{p}}(\tau) w(\tau)\,d\tau \bigg)^{\frac{q}{s - q}} [\Phi_2]^{\frac{q}{p}}(t) w(t)\,dt \bigg)^{\frac{s - q}{qs}},\\
    B_{2,2} : & = \bigg( \int_0^{\infty} W_*^{\frac{q}{s - q}}(t) \bigg(
    \int_0^t \bigg( \frac{\Phi_2(\tau)}{\Phi \big[\Psi^s \psi^{1-s};s\big](\tau)} \bigg)^{\frac{s}{s - p}}
    \phi \big[\Psi^s \psi^{1-s};s\big](\tau)\,d\tau \bigg)^{\frac{q (s - p)}{p (s - q)}}
    w(t)\,dt \bigg)^{\frac{s - q}{qs}};
    \end{align*}

    {\rm (iii)} $q < s \le p$, and in this case
    $$
    c_2 \approx B_{2,1} + C_2 + \|\| {\bf 1} \|_{p, \Psi[v;s]^{2p}u,(0,t)} \|_{q,w,(0,\infty)} /
    \| {\bf 1} \|_{s, \psi[v;s],(0,\infty)},
    $$
    where
    \begin{align*}
    C_2 : & =  \bigg( \int_0^{\infty} \bigg( \esup_{\tau \in (0,t)}
    \frac{[\Phi_2]^{\frac{s}{p}}(\tau)}{\Phi \big[\Psi^s \psi^{1-s};s\big] (\tau)}\bigg)^{\frac{q}{s - q}} W_*^{\frac{q}{s - q}}(t) w(t)\,dt \bigg)^{\frac{s - q}{s q}};
    \end{align*}

    {\rm (iv)} $s \le q < \infty$ and $s \le p$, and in this case
    $$
    c_2 =
    D_2 + \|\| {\bf 1} \|_{p, \Psi[v;s]^{2p}u,(0,t)} \|_{q,w,(0,\infty)} /
    \| {\bf 1} \|_{s, \psi[v;s],(0,\infty)},
    $$
    where
    $$
    D_2 : = \sup_{t > 0}\Phi \big[\Psi^s \psi^{1-s};s\big]^{- \frac{1}{s}}(t) \bigg( \int_0^{\infty} [\Phi_2]^{\frac{q}{p}}(\min\{\tau,t\}) w(\tau)\,d\tau \bigg)^{ \frac{1}{q} };
    $$

    {\rm (v)} $s \le p$ and $q = \infty$, and in this case $$
    c_2 = E_2 + \|\| {\bf 1} \|_{p, \Psi[v;s]^{2p}u,(0,t)} \|_{q,w,(0,\infty)} / \| {\bf 1} \|_{s, \psi[v;s],(0,\infty)},
    $$
    where
    $$
    E_2 : = \esup_{t > 0}\Phi \big[\Psi^s \psi^{1-s};s\big]^{- \frac{1}{s}}(t) \bigg( \esup_{\tau > 0} \, \Phi_2(\min\{\tau,t\}) w(\tau) \bigg)^{\frac{1}{p}};
    $$

    {\rm (vi)} $p < s$ and $q = \infty$, and in this case $$
    c_2 = F_2 + \|\| {\bf 1} \|_{p, \Psi[v;s]^{2p}u,(0,t)} \|_{q,w,(0,\infty)} / \| {\bf 1} \|_{s, \psi[v;s],(0,\infty)},
    $$
    where
    $$
    F_2 : = \esup_{t > 0} w(t) \bigg( \int_0^t \bigg( \int_{\tau}^t u(y)
    \Phi \big[\Psi^s \psi^{1-s};s\big]^{-1}(y)\,dy  \bigg)^{\frac{s}{s - p}}\phi \big[\Psi^s \psi^{1-s};s\big](\tau)\,d\tau \bigg)^{\frac{s - p}{s p}}.
    $$
\end{theorem}

\begin{proof}
    By Corollary \ref{RT.cor.main.4} (applied to $H_{p,u}$ with $\delta = 1$), inequality \eqref{IHI.2} with the best constant $c_2$ holds if and
    only if both
    \begin{equation}\label{RT.SC.thm.4.eq.2}
    \left\| H_{p,\Psi[v;s]^{2p} u}\left(\int_0^x h\right)\right\|_{q,w,(0,\infty)}
    \le c_{2,1} \,\|h\|_{s,\Psi[v;s]^{s} \psi[v;s]^{1-s},(0,\infty)}, \, h \in {\mathfrak M}^+,
    \end{equation}
    and
    \begin{equation}\label{RT.SC.thm.4.eq.3}
    \| \| {\bf 1} \|_{p, \Psi[v;s]^{2p}u,(0,t)} \|_{q,w,(0,\infty)}\le c_{2,2} \| {\bf 1} \|_{s,
        \psi[v;s],(0,\infty)},
    \end{equation}
    hold.

    Moreover, $c_2 \approx c_{2,1} + \|\| {\bf 1} \|_{p, \Psi[v;s]^{2p}u,(0,t)} \|_{q,w,(0,\infty)} /
    \| {\bf 1} \|_{s, \psi[v;s],(0,\infty)}$.

    Now the statement follows by Theorem \ref{RT.SC.thm.1.1}.
\end{proof}

We have  the following statement when $s = 1$.
\begin{theorem}\label{RT.SC.thm.1.10}
    Let $0 < p < \infty$ and $0 < q \le \infty$. Assume that $u,\,w \in {\mathcal W}(0,\infty)$
    and $v \in {\mathcal W}(0,\infty)$ be such that $V_*(x) < \infty$ for all $x > 0$. Denote by
    $$
    V_3^* (\tau) : = \int_0^{\tau} u(x) \{V_* \cdot [V_1^*]^2\}^{2p}(x)\,dx, ~ \tau > 0.
    $$
    Recall that
    $$
    V_1^* (x) : = \bigg( \int_0^x V_*^{-2}(t)v(t)\,dt \bigg)^{1/3},\qquad (x > 0).
    $$
    Then inequality        \begin{equation}\label{RT.SC.thm.1.9.eq.1}
    \left\|H_{p,u} \bigg( \int_x^{\infty} h\bigg)\right\|_{q,w,(0,\infty)} \leq c_2^1
    \,\|h\|_{1,V_*^{-1},(0,\infty)},~h\in {\mathfrak M}^+
    \end{equation}
    with the best constant $c_2^1$ holds if
    and only if:

    {\rm (i)} $p < 1 \le q < \infty$, and in this case
    $$
    c_2^1 \approx A_{2,1}^1 +
    A_{2,2}^1 + \big\|\| {\bf 1} \|_{p, V_*^{2p}u,(0,t)} \big\|_{q,w,(0,\infty)} / \| {\bf 1} \|_{1, v,(0,\infty)},
    $$
    where
    \begin{align*}
    A_{2,1}^1 : & = \sup_{t > 0}\bigg( \int_0^t [V_3^*]^{ q / p}(\tau)
    w(\tau)\,d\tau \bigg)^{1 / q} [V_1^*]^{- 1} (t),\\
    A_{2,2}^1 : & = \sup_{t > 0} W_*^{\frac{1}{q}} (t) \bigg( \int_0^t \bigg(
    \frac{V_3^*(\tau)}{V_1^*(\tau)} \bigg)^{\frac{1}{1-p}} \{V_* \cdot [V_1^*]\}^{-2}(\tau)v(\tau)\,d\tau
    \bigg)^{\frac{1 - p}{p}};
    \end{align*}

    {\rm (ii)} $q  < 1$ and $p < 1$, and in this case
    $$
    c_2^1 \approx
    B_{2,1}^1 + B_{2,2}^1 + \big\|\| {\bf 1} \|_{p, V_*^{2p}u,(0,t)} \big\|_{q,w,(0,\infty)} / \| {\bf 1} \|_{1, v,(0,\infty)},
    $$
    where
    \begin{align*}
    B_{2,1}^1 : & =  \bigg( \int_0^{\infty} [V_1^*]^{\frac{q}{q - 1}}(t)
    \bigg(\int_0^t [V_3^*]^{\frac{ q}{ p}}(\tau) w(\tau)\,d\tau \bigg)^{\frac{q}{1 - q}} [V_3^*]^{\frac{q}{p}}(t) w(t)\,dt \bigg)^{\frac{1 - q}{q}},\\
    B_{2,2}^1 : & = \bigg( \int_0^{\infty} W_*^{\frac{q}{1 - q}}(t) \bigg(
    \int_0^t \bigg( \frac{V_3^*(\tau)}{V_1^*(\tau)} \bigg)^{\frac{1}{1 - p}} \{V_* \cdot [V_1^*]\}^{-2}(\tau)v(\tau)\,d\tau \bigg)^{\frac{q (1 - p)}{p (1 - q)}} w(t)\,dt \bigg)^{\frac{1 - q}{q}};
    \end{align*}

    {\rm (iii)} $q < 1 \le p$, and in this case
    $$
    c_2^1 \approx B_{2,1}^1 + C_2^1 + \big\|\| {\bf 1} \|_{p, V_*^{2p}u,(0,t)} \big\|_{q,w,(0,\infty)} / \| {\bf 1} \|_{1, v,(0,\infty)},
    $$
    where
    \begin{align*}
    C_2^1 : & =  \bigg( \int_0^{\infty} \bigg( \esup_{\tau \in (0,t)}
    \frac{[V_3^*]^{\frac{1}{p}}(\tau)}{V_1^* (\tau)}\bigg)^{\frac{q}{1 - q}} W_*^{\frac{q}{1 - q}}(t) w(t)\,dt \bigg)^{\frac{1 - q}{q}};
    \end{align*}

    {\rm (iv)} $1 \le q < \infty$ and $1 \le p$, and in this case
    $$
    c_2^1 = D_2^1 + \big\|\| {\bf 1} \|_{p, V_*^{2p}u,(0,t)} \big\|_{q,w,(0,\infty)} / \| {\bf 1} \|_{1, v,(0,\infty)},
    $$
    where
    $$
    D_2^1 : = \sup_{t > 0} [V_1^*]^{- 1}(t) \bigg( \int_0^{\infty} [V_3^*]^{\frac{q}{p}}(\min\{\tau,t\}) w(\tau)\,d\tau \bigg)^{ \frac{1}{q}};
    $$

    {\rm (v)} $1 \le p$ and $q = \infty$, and in this case $$
    c_2^1 = E_2^1 + \big\|\| {\bf 1} \|_{p, V_*^{2p}u,(0,t)} \big\|_{q,w,(0,\infty)} / \| {\bf 1} \|_{1, v,(0,\infty)},
    $$
    where
    $$
    E_2^1 : = \esup_{t > 0}[V_1^*]^{- 1}(t) \bigg( \esup_{\tau > 0} \,    [V_3^*](\min\{\tau,t\}) w(\tau) \bigg)^{\frac{1}{p}};
    $$

    {\rm (vi)} $p < 1$ and $q = \infty$, and in this case
    $$
    c_2^1 = F_2^1 + \big\|\| {\bf 1} \|_{p, V_*^{2p}u,(0,t)} \big\|_{q,w,(0,\infty)} / \| {\bf 1} \|_{1, v,(0,\infty)},
    $$
    where
    $$
    F_2^1 : = \esup_{t > 0} w(t)^{\frac{1}{p}} \bigg( \int_0^t \bigg( \int_{\tau}^t u(y)
    [V_1^*]^{2p-1}(y)\,dy  \bigg)^{\frac{1}{1 - p}}\{V_* \cdot [V_1^*]\}^{-2}(\tau)v(\tau)\,d\tau \bigg)^{\frac{1 - p}{p}}.
    $$
\end{theorem}

\begin{proof}
    By Corollary \ref{RT.cor.main.18.0} applied to the operator $H_{p,u}$, inequality \eqref{RT.SC.thm.1.9.eq.1}
    with the best constant $c_2^1$ holds if and
    only if both
    \begin{equation}\label{RT.SC.thm.1.19.eq.2}
    \left\|\int_0^x \{V_* \cdot [V_1^*]^2\}^{2p} uf\right\|_{q/p,w,(0,\infty)}
    \le c_{2,1}^p \,\|f\|_{1/p,\{V_* \cdot [V_1^*]\}^{-2}v,(0,\infty)}, \, f \in {\mathfrak M}^{\downarrow},
    \end{equation}
    and
    \begin{equation}\label{RT.SC.thm.1.19.eq.3}
    \big\| \| {\bf 1} \|_{p, V_*^{2p}u,(0,t)} \big \|_{q,w,(0,\infty)}\le c_{2,2} \| {\bf 1} \|_{1,
        v,(0,\infty)},
    \end{equation}
    hold. Moreover, $c_2^1 \approx c_{2,1} + \big\|\| {\bf 1} \|_{p, V_*^{2p}u,(0,t)} \big\|_{q,w,(0,\infty)} / \| {\bf 1} \|_{1, v,(0,\infty)}$. Applying Theorem \ref{Thm.2.5.} we obtain the statement.
\end{proof}

For the sake of completeness we give the characterizations of
inequalities of \eqref{IHI.3} and \eqref{IHI.4} here.
\begin{theorem}\label{RT.SC.thm.1.1.3}
    Let $0 < p < \infty$, $0 < q \le \infty$ and $1 < s < \infty$. Assume that $u,\,w \in {\mathcal W}(0,\infty)$
    and $v \in {\mathcal W}(0,\infty)$ be such that \eqref{RT.thm.main.4.eq.0} holds.
    Recall that
    $$
    \Psi \big[v;s\big](x) = \bigg( \int_x^{\infty}
    v^{1-{s}^{\prime}}(t)\,dt\bigg)^{\frac{1}{s^{\prime}+ 1}}, ~ x > 0.
    $$
    Denote by
    $$
    \Psi_1(\tau) : = \int_{\tau}^{\infty} u(x) \Psi[v;s]^{2p}(x)\,dx = \int_{\tau}^{\infty} u(x) \bigg( \int_x^{\infty}
    v^{1-{s}^{\prime}}(t)\,dt\bigg)^{\frac{2p}{s^{\prime}+ 1}}\,dx, ~ \tau > 0.
    $$
    Then inequality \eqref{IHI.3} with the best constant $c_3$ holds if
    and only if:

    {\rm (i)} $p < s \le q < \infty$, and in this case $c_3 \approx A_{3,1} +
    A_{3,2}$, where
    \begin{align*}
    A_{3,1} : & = \sup_{t > 0}\bigg( \int_t^{\infty} [\Psi_1]^{ \frac{q}{p}}(\tau)
    w(\tau)\,d\tau \bigg)^{\frac{1}{q}} \Psi[v;s]^{- \frac{1}{s}} (t),\\
    A_{3,2} : & = \sup_{t > 0} W^{\frac{1}{q}} (t) \bigg( \int_t^{\infty} \bigg(
    \frac{\Psi_1(\tau)}{\Psi[v;s](\tau)} \bigg)^{\frac{s}{s-p}} \psi[v;s](\tau)\,d\tau
    \bigg)^{\frac{s - p}{p s}};
    \end{align*}

    {\rm (ii)} $q  < s < \infty$ and $p < s$, and in this case $c_3 \approx
    B_{3,1} + B_{3,2}$, where
    \begin{align*}
    B_{3,1} : & =  \bigg( \int_0^{\infty} \Psi[v;s]^{\frac{q}{q - s}}(t)
    \bigg(\int_t^{\infty} [\Psi_1]^{ \frac{q}{p}}(\tau) w(\tau)\,d\tau \bigg)^{\frac{q}{s - q}} [\Psi_1]^{ \frac{q}{p}}(t) w(t)\,dt \bigg)^{\frac{s - q}{qs}},\\
    B_{3,2} : & = \bigg( \int_0^{\infty} W^{\frac{q}{s - q}}(t) \bigg(
    \int_t^{\infty} \bigg( \frac{\Psi_1(\tau)}{\Psi[v;s](\tau)} \bigg)^{\frac{s}{s - p}} \psi[v;s](\tau)\,d\tau \bigg)^{\frac{q (s - p)}{p (s - q)}} w(t)\,dt \bigg)^{\frac{s-q}{qs}};
    \end{align*}

    {\rm (iii)} $q < s \le p$, and in this case $c_3 \approx B_{3,1} + C_3$,
    where
    \begin{align*}
    C_3 : & =  \bigg( \int_0^{\infty} \bigg( \esup_{\tau \in (t,\infty)}
    \frac{[\Psi_1]^{\frac{s}{p}}(\tau)}{\Psi[v;s] (\tau)}\bigg)^{\frac{q}{s - q}} W^{\frac{q}{s - q}}(t) w(t)\,dt \bigg)^{\frac{s - q}{s q}};
    \end{align*}

    {\rm (iv)} $s \le q < \infty$ and $s \le p$, and in this case $c_3 =
    D_3$, where
    $$
    D_3 : = \sup_{t > 0}\Psi[v;s]^{- \frac{1}{s}}(t) \bigg( \int_0^{\infty} [\Psi_1]^{\frac{q}{p}}(\max\{\tau,t\}) w(\tau)\,d\tau \bigg)^{\frac{1}{q}};
    $$

    {\rm (v)} $s \le p$ and $q = \infty$, and in this case $c_3 = E_3$,
    where
    $$
    E_3 : = \esup_{t > 0}\Psi[v;s]^{- \frac{1}{s}}(t) \bigg( \esup_{\tau > 0} \,    \Psi_1(\max\{\tau,t\}) w(\tau) \bigg)^{\frac{1}{p}};
    $$

    {\rm (vi)} $p < s$ and $q = \infty$, and in this case $c_3 = F_3$, where
    $$
    F_3 : = \esup_{t > 0} w(t) \bigg( \int_t^{\infty} \bigg( \int_t^{\tau} u(y)
    \Psi[v;s]^{-1}(y)\,dy  \bigg)^{\frac{s}{s - p}}\psi[v;s](\tau)\,d\tau \bigg)^{\frac{s - p}{s p}}.
    $$
\end{theorem}

\begin{proof}
    By change of variables $x = 1 / t$, it is easy to see that inequality \eqref{IHI.3} holds if and only if
    \begin{equation*}
    \left\|H_{p,\tilde{u}} \bigg( \int_0^x h\bigg)\right\|_{q,\tilde{w},(0,\infty)} \leq c
    \,\|h\|_{s,\tilde{v},(0,\infty)}
    \end{equation*}
    holds, where
    $$
    \tilde{u} (t) = u \bigg(\frac{1}{t}\bigg)\frac{1}{t^2}, ~ \tilde{w} (t) = w \bigg(\frac{1}{t}\bigg)\frac{1}{t^2}, ~\tilde{v} (t) = v \bigg(\frac{1}{t}\bigg)\bigg(\frac{1}{t^2}\bigg)^{1-s}, ~ t > 0,
    $$
    when $0 < q < \infty$, and
    $$
    \tilde{u} (t) = u \bigg(\frac{1}{t}\bigg)\frac{1}{t^2}, ~ \tilde{w} (t) = w \bigg(\frac{1}{t}\bigg), ~\tilde{v} (t) = v \bigg(\frac{1}{t}\bigg)\bigg(\frac{1}{t^2}\bigg)^{1-s}, ~ t > 0,
    $$
    when $q = \infty$.

    Using Theorem \ref{RT.SC.thm.1.1}, and then applying substitution of variables mentioned above three times, we get the statement.
\end{proof}

\begin{theorem}\label{RT.SC.thm.1.3.1}
    Let $0 < p < \infty$ and $0 < q \le \infty$. Assume that $u,\,w \in {\mathcal W}(0,\infty)$
    and $v \in {\mathcal W}(0,\infty)$ be such that $V_*(x) < \infty$ for all $x > 0$. Denote by
    $$
    V_2^* (\tau) : = \int_{\tau}^{\infty} u(x) V_*^{2p}(x)\,dx, ~ \tau > 0.
    $$
    Then inequality
    \begin{equation}\label{RT.SC.thm.1.2.eq.1.1}
    \left\|H_{p,u}^* \bigg( \int_x^{\infty} h\bigg)\right\|_{q,w,(0,\infty)} \leq c_3^1
    \,\|h\|_{1,V_*^{-1},(0,\infty)},~h\in {\mathfrak M}^+
    \end{equation}
    with the best constant $c_3^1$ holds if
    and only if:

    {\rm (i)} $p < 1 \le q < \infty$, and in this case $c_3^1 \approx A_{3,1}^1 +
    A_{3,2}^1$, where
    \begin{align*}
    A_{3,1}^1 : & = \sup_{t > 0}\bigg( \int_t^{\infty} [V_2^*]^{\frac{q}{p}}(\tau)
    w(\tau)\,d\tau \bigg)^{\frac{1}{q}} V_*^{- 1} (t),\\
    A_{3,2}^1 : & = \sup_{t > 0} W^{\frac{1}{q}} (t) \bigg( \int_t^{\infty} \bigg(
    \frac{V_2^*(\tau)}{V_*(\tau)} \bigg)^{\frac{1}{1-p}} v(\tau)\,d\tau
    \bigg)^{\frac{1 - p}{p}};
    \end{align*}

    {\rm (ii)} $q  < 1$ and $p < 1$, and in this case $c_3^1 \approx
    B_{3,1}^1 + B_{3,2}^1$, where
    \begin{align*}
    B_{3,1}^1 : & =  \bigg( \int_0^{\infty} V_*^{\frac{q}{q - 1}}(t)
    \bigg(\int_t^{\infty} [V_2^*]^{\frac{q}{p}}(\tau) w(\tau)\,d\tau \bigg)^{\frac{q}{1 - q}} [V_2^*]^{\frac{q}{p}}(t) w(t)\,dt \bigg)^{\frac{1 - q}{q}},\\
    B_{3,2}^1 : & = \bigg( \int_0^{\infty} W^{\frac{q}{1 - q}}(t) \bigg(
    \int_t^{\infty} \bigg( \frac{V_2^*(\tau)}{V_*(\tau)} \bigg)^{\frac{1}{1 - p}} v(\tau)\,d\tau \bigg)^{\frac{q (1 - p)}{p (1 - q)}} w(t)\,dt \bigg)^{\frac{1 - q}{q}};
    \end{align*}

    {\rm (iii)} $q < 1 \le p$, and in this case $c_3^1 \approx B_{3,1}^1 + C_3^1$,
    where
    \begin{align*}
    C_3^1 : & =  \bigg( \int_0^{\infty} \bigg( \esup_{\tau \in (t,\infty)}
    \frac{[V_2^*]^{\frac{1}{p}}(\tau)}{V_* (\tau)}\bigg)^{\frac{q}{1 - q}} W^{\frac{q}{1 - q}}(t) w(t)\,dt \bigg)^{\frac{1 - q}{q}};
    \end{align*}

    {\rm (iv)} $1 \le q < \infty$ and $1 \le p$, and in this case $c_3^1 = D_3^1$, where
    $$
    D_3^1 : = \sup_{t > 0} V_*^{- 1}(t) \bigg( \int_0^{\infty} [V_2^*]^{\frac{q}{p}}(\max\{\tau,t\}) w(\tau)\,d\tau \bigg)^{ \frac{1}{q}};
    $$

    {\rm (v)} $1 \le p$ and $q = \infty$, and in this case $c_3^1 = E_3^1$,
    where
    $$
    E_3^1 : = \esup_{t > 0}V_*^{- 1}(t) \bigg( \esup_{\tau > 0} \,    V_2^*(\max\{\tau,t\}) w(\tau) \bigg)^{\frac{1}{p}};
    $$

    {\rm (vi)} $p < 1$ and $q = \infty$, and in this case $c_3^1 = F_3^1$, where
    $$
    F_3^1 : = \esup_{t > 0} w(t)^{\frac{1}{p}} \bigg( \int_t^{\infty} \bigg( \int_t^{\tau} u(y)
    V_*^{2p-1}(y)\,dy  \bigg)^{\frac{1}{1 - p}}v(\tau)\,d\tau \bigg)^{\frac{1 - p}{p}}.
    $$
\end{theorem}

\begin{proof}
    By change of variables $x = 1 / t$, it is easy to see that inequality \eqref{RT.SC.thm.1.2.eq.1.1} holds if and only if
    \begin{equation*}
    \left\|H_{p,\tilde{u}} \bigg( \int_0^x h\bigg)\right\|_{q,\tilde{w},(0,\infty)} \leq c
    \,\|h\|_{1,\tilde{V}^{-1},(0,\infty)},~h\in {\mathfrak M}^+
    \end{equation*}
    holds, where
    $$
    \tilde{u} (t) = u \bigg(\frac{1}{t}\bigg)\frac{1}{t^2}, ~ \tilde{w} (t) = w \bigg(\frac{1}{t}\bigg)\frac{1}{t^2}, ~\tilde{V} (t) = \int_0^t v \bigg(\frac{1}{y}\bigg)\frac{1}{y^2}\,dy, ~ t > 0,
    $$
    when $0 < q < \infty$, and
    $$
    \tilde{u} (t) = u \bigg(\frac{1}{t}\bigg)\frac{1}{t^2}, ~ \tilde{w} (t) = w \bigg(\frac{1}{t}\bigg), ~\tilde{V} (t) = \int_0^t v \bigg(\frac{1}{y}\bigg)\frac{1}{y^2}\,dy, ~ t > 0,
    $$
    when $q = \infty$.

    Applying Theorem \ref{RT.SC.thm.1.3}, and then using substitution of variables mentioned above three times, we get the statement.
\end{proof}

\begin{theorem}\label{RT.SC.thm.2.2.1}
    Let $0 < p < \infty$, $0 < q \le \infty$ and $1 < s < \infty$. Assume that $u,\,w \in {\mathcal W}(0,\infty)$
    and $v \in {\mathcal W}(0,\infty)$ be such that \eqref{RT.thm.main.3.eq.0} holds. Denote by
    $$
    \Psi_2(\tau) : = \int_{\tau}^{\infty} u(x) \bigg(\Phi\big[v;s\big] \cdot \Psi \big[\Phi[v;s]^{s} \phi[v;s]^{1-s};s\big]\bigg)^{2p}(x)\,dx, ~ \tau > 0.
    $$
    Recall that
    \begin{align*}
    \Phi \big[v;s\big](x) & = \bigg( \int_0^x
    v^{1-{s}^{\prime}}(t)\,dt\bigg)^{\frac{1}{s^{\prime} + 1}},~ x > 0, \\
    \psi \big[\Phi[v;s]^{s} \phi[v;s]^{1-s};s\big] (x) & \\
    & \hspace{-3cm} \approx \bigg\{\int_x^{\infty} \bigg( \int_0^t v^{1 - s'}\bigg)^{-\frac{2s'}{1 + s'}} v^{1-s'}(t)\,dt\bigg\}^{- \frac{s'}{1 + s'}} \bigg( \int_0^x v^{1 - s'}\bigg)^{- \frac{2s'}{1 + s'}} v^{1-s'}(x),\\
    \Psi \big[\Phi[v;s]^{s} \phi[v;s]^{1-s};s\big]  (x)
    & \approx \bigg\{\int_x^{\infty} \bigg( \int_0^t v^{1 - s'}\bigg)^{- \frac{2s'}{1 + s'}} v^{1-s'}(t)\,dt\bigg\}^{\frac{1}{1 + s'}},
    \end{align*}
    Then inequality \eqref{IHI.4} with the best constant $c_4$ holds if
    and only if:

    {\rm (i)} $p < s \le q < \infty$, and in this case
    $$
    c_4 \approx A_{4,1} +
    A_{4,2} + \|\| {\bf 1} \|_{p, \Phi[v;s]^{2p}u,(t,\infty)} \|_{q,w,(0,\infty)} /
    \| {\bf 1} \|_{s, \phi[v;s],(0,\infty)},
    $$
    where
    \begin{align*}
    A_{4,1} : & = \sup_{t > 0}\bigg( \int_t^{\infty} [\Psi_2]^{\frac{q}{p}}(\tau)
    w(\tau)\,d\tau \bigg)^{\frac{1}{q}} \Psi \big[\Phi^s \phi^{1-s};s\big]^{- \frac{1}{s}} (t),\\
    A_{4,2} : & = \sup_{t > 0} W^{\frac{1}{q}} (t) \bigg( \int_t^{\infty} \bigg(
    \frac{\Psi_2(\tau)}{\Psi \big[\Phi^s \phi^{1-s};s\big](\tau)} \bigg)^{\frac{s}{s-p}} \phi \big[\Phi^s \phi^{1-s};s\big](\tau)\,d\tau
    \bigg)^{\frac{s - p}{p s}};
    \end{align*}

    {\rm (ii)} $q  < s < \infty$ and $p < s$, and in this case
    $$
    c_4 \approx
    B_{4,1} + B_{4,2} + \|\| {\bf 1} \|_{p, \Phi[v;s]^{2p}u,(t,\infty)} \|_{q,w,(0,\infty)} /
    \| {\bf 1} \|_{s, \phi[v;s],(0,\infty)},
    $$
    where
    \begin{align*}
    B_{4,1} : & =  \bigg( \int_0^{\infty} \Psi \big[\Phi^s \phi^{1-s};s\big]^{\frac{q}{q - s}}(t)
    \bigg(\int_t^{\infty} [\Psi_2]^{\frac{q}{p}}(\tau) w(\tau)\,d\tau \bigg)^{\frac{q}{s - q}} [\Psi_2]^{ \frac{q}{p}}(t) w(t)\,dt \bigg)^{\frac{s - q}{qs}},\\
    B_{4,2} : & = \bigg( \int_0^{\infty} W^{\frac{q}{s - q}}(t) \bigg(
    \int_t^{\infty} \bigg( \frac{\Psi_2(\tau)}{\Psi \big[\Phi^s \phi^{1-s};s\big](\tau)} \bigg)^{\frac{s}{s - p}}
    \psi \big[\Phi^s \phi^{1-s};s\big](\tau)\,d\tau \bigg)^{\frac{q (s - p)}{p (s - q)}}w(t)\,dt \bigg)^{\frac{s - q}{qs}};
    \end{align*}

    {\rm (iii)} $q < s \le p$, and in this case
    $$
    c_4 \approx B_{4,1} + C_4 + \|\| {\bf 1} \|_{p, \Phi[v;s]^{2p}u,(t,\infty)} \|_{q,w,(0,\infty)} /
    \| {\bf 1} \|_{s, \phi[v;s],(0,\infty)},
    $$
    where
    \begin{align*}
    C_4 : & =  \bigg( \int_0^{\infty} \bigg( \esup_{\tau \in (t,\infty)}
    \frac{[\Psi_2]^{\frac{s}{p}}(\tau)}{\Psi \big[\Phi^s \phi^{1-s};s\big] (\tau)}\bigg)^{\frac{q}{s - q}} W^{\frac{q}{s - q}}(t) w(t)\,dt \bigg)^{\frac{s - q}{s q}};
    \end{align*}

    {\rm (iv)} $s \le q < \infty$ and $s \le p$, and in this case
    $$
    c_4 =
    D_4 + \|\| {\bf 1} \|_{p, \Phi[v;s]^{2p}u,(t,\infty)} \|_{q,w,(0,\infty)} /
    \| {\bf 1} \|_{s, \phi[v;s],(0,\infty)},
    $$
    where
    $$
    D_4 : = \sup_{t > 0}\Psi \big[\Phi^s \phi^{1-s};s\big]^{- \frac{1}{s}}(t) \bigg( \int_0^{\infty} [\Psi_2]^{\frac{q}{p}}(\max\{\tau,t\}) w(\tau)\,d\tau \bigg)^{ \frac{1}{q}};
    $$

    {\rm (v)} $s \le p$ and $q = \infty$, and in this case $$
    c_4 = E_4 + \|\| {\bf 1} \|_{p, \Phi[v;s]^{2p}u,(t,\infty)} \|_{q,w,(0,\infty)} / \| {\bf 1} \|_{s, \phi[v;s],(0,\infty)},
    $$
    where
    $$
    E_4 : = \esup_{t > 0}\Psi \big[\Phi^s \phi^{1-s};s\big]^{- \frac{1}{s}}(t) \bigg( \esup_{\tau > 0} \, \Psi_2(\max\{\tau,t\}) w(\tau) \bigg)^{\frac{1}{p}};
    $$

    {\rm (vi)} $p < s$ and $q = \infty$, and in this case $$
    c_4 = F_4 + \|\| {\bf 1} \|_{p, \Phi[v;s]^{2p}u,(t,\infty)} \|_{q,w,(0,\infty)} / \| {\bf 1} \|_{s, \phi[v;s],(0,\infty)},
    $$
    where
    $$
    F_4 : = \esup_{t > 0} w(t) \bigg( \int_t^{\infty} \bigg( \int_t^{\tau} u(y)
    \Psi \big[\Phi^s \phi^{1-s};s\big]^{-1}(y)\,dy  \bigg)^{\frac{s}{s - p}}\psi \big[\Phi^s \phi^{1-s};s\big](\tau)\,d\tau \bigg)^{\frac{s - p}{s p}}.
    $$
\end{theorem}

\begin{proof}
    Obviously, inequality \eqref{IHI.4} holds if and only if
    \begin{equation*}
    \left\|H_{p,\tilde{u}} \bigg( \int_x^{\infty} h\bigg)\right\|_{q,\tilde{w},(0,\infty)} \leq c
    \,\|h\|_{s,\tilde{v},(0,\infty)}
    \end{equation*}
    holds, where
    $$
    \tilde{u} (t) = u \bigg(\frac{1}{t}\bigg)\frac{1}{t^2}, ~ \tilde{w} (t) = w \bigg(\frac{1}{t}\bigg)\frac{1}{t^2}, ~\tilde{v} (t) = v \bigg(\frac{1}{t}\bigg)\bigg(\frac{1}{t^2}\bigg)^{1-s}, ~ t > 0,
    $$
    when $0 < q < \infty$, and
    $$
    \tilde{u} (t) = u \bigg(\frac{1}{t}\bigg)\frac{1}{t^2}, ~ \tilde{w} (t) = w \bigg(\frac{1}{t}\bigg), ~\tilde{v} (t) = v \bigg(\frac{1}{t}\bigg)\bigg(\frac{1}{t^2}\bigg)^{1-s}, ~ t > 0,
    $$
    when $q = \infty$.

    Using Theorem \ref{RT.SC.thm.2.2}, and then applying substitution of variables mentioned above three times, we get the statement.
\end{proof}

\begin{theorem}\label{RT.SC.thm.1.10.1}
    Let $0 < p < \infty$ and $0 < q \le \infty$. Assume that $u,\,w \in {\mathcal W}(0,\infty)$
    and $v \in {\mathcal W}(0,\infty)$ be such that $V(x) < \infty$ for all $x > 0$.
    Recall that
    $$
    V_1 (x) : = \bigg( \int_x^{\infty} V^{-2}v \bigg)^{\frac{1}{3}}, \qquad (x > 0).
    $$
    Denote by
    $$
    V_3 (\tau) : = \int_{\tau}^{\infty} u(x) \{V \cdot V_1^2\}^{2p}(x)\,dx, ~ \tau > 0.
    $$
    Then inequality
    \begin{equation}\label{RT.SC.thm.1.9.eq.1.1}
    \left\|H_{p,u}^* \bigg( \int_0^x h\bigg)\right\|_{q,w,(0,\infty)} \leq c_4^1
    \,\|h\|_{1,V^{-1},(0,\infty)},
    \end{equation}
    with the best constant $c_4^1$ holds if
    and only if:

    {\rm (i)} $p < 1 \le q < \infty$, and in this case
    $$
    c_4^1 \approx A_{4,1}^1 +
    A_{4,2}^1 + \|\| {\bf 1} \|_{p, V^{2p}u,(t,\infty)} \|_{q,w,(0,\infty)} / \| {\bf 1} \|_{1, v,(0,\infty)},
    $$
    where
    \begin{align*}
    A_{4,1}^1 : & = \sup_{t > 0}\bigg( \int_t^{\infty} [V_3]^{ \frac{q}{p}}(\tau)
    w(\tau)\,d\tau \bigg)^{\frac{1}{q}} [V_1]^{- 1} (t),\\
    A_{4,2}^1 : & = \sup_{t > 0} W^{\frac{1}{q}} (t) \bigg( \int_t^{\infty} \bigg(
    \frac{V_3(\tau)}{V_1(\tau)} \bigg)^{\frac{1}{1-p}} \{V \cdot [V_1]\}^{-2}(\tau)v(\tau)\,d\tau\bigg)^{\frac{1-p}{p}};
    \end{align*}

    {\rm (ii)} $q  < 1$ and $p < 1$, and in this case
    $$
    c_4^1 \approx
    B_{4,1}^1 + B_{4,2}^1 + \|\| {\bf 1} \|_{p, V^{2p}u,(t,\infty)} \|_{q,w,(0,\infty)} / \| {\bf 1} \|_{1, v,(0,\infty)},
    $$
    where
    \begin{align*}
    B_{4,1}^1 : & =  \bigg( \int_0^{\infty} V_1^{\frac{q}{q - 1}}(t)
    \bigg(\int_t^{\infty} [V_3]^{\frac{q}{p}}(\tau) w(\tau)\,d\tau \bigg)^{\frac{q}{1 - q}} V_3^{\frac{q}{p}}(t) w(t)\,dt \bigg)^{\frac{1-q}{q}},\\
    B_{4,2}^1 : & = \bigg( \int_0^{\infty} W^{\frac{q}{1 - q}}(t) \bigg(
    \int_t^{\infty} \bigg( \frac{V_3(\tau)}{V_1(\tau)} \bigg)^{\frac{1}{1 - p}} \{V \cdot [V_1]\}^{-2}(\tau)v(\tau)\,d\tau \bigg)^{\frac{q (1 - p)}{p (1 - q)}} w(t)\,dt \bigg)^{\frac{1-q}{q}};
    \end{align*}

    {\rm (iii)} $q < 1 \le p$, and in this case
    $$
    c_4^1 \approx B_{4,1}^1 + C_4^1 + \|\| {\bf 1} \|_{p, V^{2p}u,(t,\infty)} \|_{q,w,(0,\infty)} / \| {\bf 1} \|_{1, v,(0,\infty)},
    $$
    where
    \begin{align*}
    C_4^1 : & =  \bigg( \int_0^{\infty} \bigg( \esup_{\tau \in (t,\infty)}
    \frac{[V_3]^{\frac{1}{p}}(\tau)}{V_1 (\tau)}\bigg)^{\frac{q}{1-q}} W^{\frac{q}{1-q}}(t) w(t)\,dt \bigg)^{\frac{1-q}{q}};
    \end{align*}

    {\rm (iv)} $1 \le q < \infty$ and $1 \le p$, and in this case
    $$
    c_4^1 = D_4^1 + \|\| {\bf 1} \|_{p, V^{2p}u,(t,\infty)} \|_{q,w,(0,\infty)} / \| {\bf 1} \|_{1, v,(0,\infty)},
    $$
    where
    $$
    D_4^1 : = \sup_{t > 0} V_1^{- 1}(t) \bigg( \int_0^{\infty} [V_3]^{\frac{q}{p}}(\max\{\tau,t\}) w(\tau)\,d\tau \bigg)^{ \frac{1}{q}};
    $$

    {\rm (v)} $1 \le p$ and $q = \infty$, and in this case $$
    c_4^1 = E_4^1 + \|\| {\bf 1} \|_{p, V^{2p}u,(t,\infty)} \|_{q,w,(0,\infty)} / \| {\bf 1} \|_{1, v,(0,\infty)},
    $$
    where
    $$
    E_4^1 : = \esup_{t > 0}V_1^{- 1}(t) \bigg( \esup_{\tau > 0} \,    [V_3](\max\{\tau,t\}) w(\tau) \bigg)^{\frac{1}{p}};
    $$

    {\rm (vi)} $p < 1$ and $q = \infty$, and in this case
    $$
    c_4^1 = F_4^1 + \|\| {\bf 1} \|_{p, V^{2p}u,(t,\infty)} \|_{q,w,(0,\infty)} / \| {\bf 1} \|_{1, v,(0,\infty)},
    $$
    where
    $$
    F_4^1 : = \esup_{t > 0} w(t)^{\frac{1}{p}} \bigg( \int_t^{\infty} \bigg( \int_t^{\tau} u(y)
    V_1^{2p-1}(y)\,dy  \bigg)^{\frac{1}{1-p}}\{V \cdot V_1\}^{-2}(\tau)v(\tau)\,d\tau \bigg)^{\frac{1-p}{p}}.
    $$
\end{theorem}

\begin{proof}
    Obviously, inequality \eqref{RT.SC.thm.1.9.eq.1.1} holds if and only if
    \begin{equation*}
    \left\|H_{p,\tilde{u}} \bigg( \int_x^{\infty} h\bigg)\right\|_{q,\tilde{w},(0,\infty)} \leq c
    \,\|h\|_{1,\tilde{V}_*^{-1},(0,\infty)},~h\in {\mathfrak M}^+
    \end{equation*}
    holds, where
    $$
    \tilde{u} (t) = u \bigg(\frac{1}{t}\bigg)\frac{1}{t^2}, ~ \tilde{w} (t) = w \bigg(\frac{1}{t}\bigg)\frac{1}{t^2}, ~\tilde{V}_* (t) = \int_t^{\infty} v \bigg(\frac{1}{y}\bigg)\frac{1}{y^2}\,dy, ~ t > 0,
    $$
    when $0 < q < \infty$, and
    $$
    \tilde{u} (t) = u \bigg(\frac{1}{t}\bigg)\frac{1}{t^2}, ~ \tilde{w} (t) = w \bigg(\frac{1}{t}\bigg), ~\tilde{V}_* (t) = \int_t^{\infty} v \bigg(\frac{1}{y}\bigg)\frac{1}{y^2}\,dy, ~ t > 0,
    $$
    when $q = \infty$.

    Applying Theorem \ref{RT.SC.thm.1.10}, and then using substitution of variables mentioned above three times, we get the statement.
\end{proof}

\begin{remark}
    It is worth to mention that Theorem \ref{RT.SC.thm.2.2} - \ref{RT.SC.thm.1.10.1} can be proved by reducing corresponding iterated inequality to the cone of monotone functions. For instance: inequality \eqref{IHI.3} with the best constant $c_3$ holds if
    and only if inequality
    $$
    \bigg\| \int_0^x f u \Psi[v;s]^{2p} \bigg\|_{{q/p}, w, (0,\infty)} \le c_3^p \,\|
    f \|_{{s / p}, \psi[v;s],(0,\infty)}, \, f \in {\mathfrak M}^{\uparrow}
    $$
    holds, and the statement of Theorem \ref{RT.SC.thm.1.1.3} immediately follows by Theorem \ref{Thm.2.5.00}.
\end{remark}

\end{document}